\documentclass{article}
\usepackage[cp1252]{inputenc}
\usepackage[OT1]{fontenc}
\usepackage[english]{babel}
\usepackage{amsmath}
\usepackage{amsfonts}
\usepackage{amssymb}
\usepackage{amsthm}
\usepackage[left=2cm,right=2cm,top=2cm,bottom=2cm]{geometry}
\usepackage{enumitem,multicol}
\usepackage{float, graphicx}

\usepackage[matrix,arrow,curve]{xy}
\usepackage{lscape,comment}

\numberwithin{equation}{section}

\def\R{\mathbb{R}}

\newtheorem{theorem}{Theorem}
\newtheorem*{remark}{Remark}
\newtheorem{lemma}{Lemma}
\newtheorem{corollary}{Corollary}
\newtheorem{definition}{Definition}
\newtheorem{prop}{Proposition}

\def\const{\operatorname{const}}
\def\spann{\operatorname{span}}
\def\sgn{\operatorname{sgn}}
\def\Id{\operatorname{Id}}
\def\Lip{\operatorname{Lip}}
\def\card{\operatorname{card}}
\def\Exp{\operatorname{Exp}}
\def\cut{\operatorname{cut}}

\newcommand{\hf}{\mathfrak{h}}

\newcommand{\Cf}{\mathfrak{C}}

\newcommand{\eq}[1]{$(\protect\ref{#1})$}
\newcommand{\be}[1]{\begin{equation}\label{#1}}
\newcommand{\ee}{\end{equation}}

\newcommand{\twofiglabel}[6]
{
\begin{figure}[htbp]
\includegraphics[width=0.47\textwidth]{#1}
\hfill
\includegraphics[width=0.47\textwidth]{#4}
\\
\parbox[t]{0.45\textwidth}{\caption{#2}\label{#3}}
\hfill
\parbox[t]{0.45\textwidth}{\caption{#5}\label{#6}}
\end{figure}
}

\renewcommand\Vec{\operatorname{Vec}}

\graphicspath{{figures/}}

\title{A sub-Finsler problem on the Cartan group\footnote{Sections 1--3 of the paper are written by E. Le Donne, 
and Sections 4--7  are written by A. Ardentov  and Yu. Sachkov. 
The work of A. Ardentov  and Yu. Sachkov is supported by the Russian Science Foundation 
under grant 17-11-01387 and performed in Ailamazyan Program Systems Institute 
of Russian Academy of Sciences.
E. Le Donne was partially supported by the Academy of Finland (grant
288501
`\emph{Geometry of subRiemannian groups}')
and by the European Research Council
 (ERC Starting Grant 713998 GeoMeG `\emph{Geometry of Metric Groups}').}}

\author{A.Ardentov\footnote{Program Systems Institute, Pereslavl-Zalessky, Russia, \tt{aaa@pereslavl.ru}}, E. Le Donne\footnote{Department of Mathematics and Statistics, P.O. Box 35,
FI-40014, University of Jyv\"askyl\"a, Finland,
\tt{ledonne@msri.org}}, 
Yu. Sachkov\footnote{Program Systems Institute, Pereslavl-Zalessky, Russia, \tt{yusachkov@gmail.com}}}

\begin{document}

\maketitle

\begin{abstract}
In this paper we study a sub-Finsler geometric problem on the free-nilpotent group of rank 2 and step 3. 
Such a group is also called Cartan group and has a natural structure of Carnot group, which we metrize considering the $\ell_\infty$ norm on its first layer.
We adopt the point of view of   time-optimal control theory.
We characterize extremal curves via Pontryagin maximum principle. We describe abnormal and singular arcs, and construct the bang-bang flow. 
\end{abstract}


\section{Introduction}

Sub-Finsler geometry is a natural generalization of Finsler geometry, sub-Riemannian geometry, and hence Riemannian geometry. 
Sub-Finsler structures   appear in geometric group theory, in the theory of 
isometrically homogeneous geodesic spaces, and
  in different applications in control theory; see 
 \cite{pansu, Breuillard-LeDonne1},  \cite{b1, b2, LeDonne_characterization}, and
   \cite{boscain3level}, respectively.

Of a special importance is the fact that Lie groups equipped with sub-Finsler structures appear in geometric group theory as asymptotic cones of nilpotent finitely generated groups. 
Indeed, in \cite{pansu} Pansu established that the asymptotic cones of  finitely generated 
nilpotent discrete groups equipped with word metrics are Carnot groups equipped with left-invariant 
sub-Finsler metric.
We remark that such metrics come from structures that are never sub-Riemannian since the norms are characterized by convex hulls of finitely many points.
Hence,
 the typical example is the $\ell_{1}$  norm. Notice that in rank-2 groups, this   norm  differs from the $\ell_{\infty}$ just by a change of variable.

Our paper gives a contribution towards the 
understanding of the geometry 
of $\ell_{\infty}$ sub-Finsler spaces.
Some natural problems are the regularity  of spheres and of geodesics.
For instance,  an unsolved problem is whether 
any pair of points  
can always be connected by a piecewise smooth
 length-minimizing curve.
 If this is the case, 
 one would like to know if the number of such pieces is
 uniformly bounded.
These are fundamental questions coming directly from the asymptotic study of nilpotent finitely generated groups.
Indeed, there are conjectures about asymptotic expansions for the volume
growth  of balls of large radii that are related to  
the rectifiability of spheres and to 
the above-mentioned regularity of geodesics 
 for the 
asymptotic cone, see \cite{Breuillard-LeDonne1}.

The problem of finding  length-minimizing curves in  sub-Finsler Lie groups    can be   reformulated as a 
time-optimal 
control problem for a system that is linear in the controls.
A formal introduction is given in \cite{BBLDS}.  
In particular, the existence of time-minimizers  is a classical consequence of Filippov's theorem. However, there are no general regularity results, except the recent paper \cite{Hakavuori_LeDonne_2018}.

The purpose of this paper is to consider a specific group and in it characterize extremal curves.

\bigskip

We study the unique $\ell_\infty$ sub-Finsler geometric problem on the free-nilpotent group of rank 2 and step 3. 
Such a group is also called Cartan group and has a natural structure of Carnot group.
In coordinates its  distribution can be expressed   by the span of two vector fields $X_1$, $X_2$.
We consider the $\ell_\infty$   norm with respect to   $X_1$, $X_2$.

The paper has the following structure. In Sec. 2 we state the problem and notice existence of minimizers. In Sec.~3 we apply Pontryagin maximum principle to the problem. In Sec. 4 we describe optimal abnormal trajectories. Further, in Sec. 5 we define different types of normal extremal arcs: bang-bang, singular, and mixed ones. 

In Sec. 6 we describe singular arcs; all singular trajectories are shown to be optimal. Moreover, we describe the fix-time attainable set via singular trajectories (this set coincides with the part of the sub-Finsler sphere filled by singular trajectories). We obtain explicit description of this set and prove that it is semi-algebraic.   

In Sec. 7 we study bang-bang trajectories. We describe the phase portrait of the Hamiltonian system corresponding to bang-bang trajectories and construct the bang-bang flow that generates these trajectories.

Finally, in Sec. 8 we discuss questions for further research.
 

We remark that    the nilpotency of the group
simplifies considerably the problem. For example, in our case bang-bang trajectories  have piecewise polynomial coordinates.

\bigskip

We mention that there are a few other works that consider the view point of sub-Finsler geometry.
Apart from the previously mentioned ones,
in the papers
\cite{clellandmoseley06, clellandmoseley07}
the authors study the sub-Finsler geometry, such as geodesics and 
rigid curves,
in three-dimensional manifolds and in 
 Engel-type manifolds. However, in those papers 
 there is an assumption that is classical in Finsler geometry:
 the norm is assumed to be smooth outside the zero section and strongly convex. 
 The present paper deals with the case where these assumptions are not satisfied. 
Another notable paper is
 \cite{cowlingmartini13}, in which the authors study the 
sub-Finsler  geometry associated with the
solutions of evolution equations given by first-order differential operators, providing one more setting where sub-Finlser geometry appears naturally.

\section{Problem statement. Existence of solutions}
Consider the 5-dimensional nilpotent Lie algebra $L = \spann (X_1, \dots, X_5)$ with the nonzero brackets
\begin{equation}
\label{b2}
[X_1, X_2] = X_3, \quad [X_1, X_3] = X_4, \quad [X_2, X_3] = X_5,
\end{equation}
whose multiplication table \eqref{b2} is depicted in Fig.~\ref{fig:b2}.
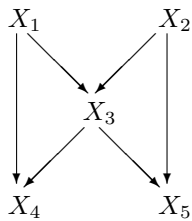
\begin{figure}[h]
\setlength{\unitlength}{1cm}

\begin{center}
\begin{picture}(4, 4)(0, 1)
\put(1.15, 3.9){ \vector(1, -1){0.8}}
\put(1, 3.9){ \vector(0, -1){2}}
\put(2.85, 3.9){ \vector(-1, -1){0.8}}
\put(3, 3.9){ \vector(0, -1){2}}
\put(1.9, 2.65){ \vector(-1, -1){0.8}}
\put(2.1, 2.65){ \vector(1, -1){0.8}}


\put(1, 1.5) {$X_4$}
\put(3, 1.5) {$X_5$}
\put(1, 3.98) {$X_1$}
\put(3, 3.98) {$X_2$}
\put(2, 2.75) {$X_3$}

\end{picture}
\end{center}
\caption{Cartan algebra}
\label{fig:b2}
\end{figure}
The Lie algebra $L$ is the free nilpotent Lie algebra of step 3 with 2 generators, it is
called the Cartan algebra.
Further, let $M$ be the connected simply connected Lie group with the Lie algebra $L$; $M$ is called the Cartan group. We will use the following model:
$$ M = \mathbb{R}^5_{x,y,z,v,w}, $$
with the Lie algebra $L$ modeled by left-invariant vector fields on $\mathbb{R}^8$:
\begin{align*}
&X_1 = \frac{\partial}{\partial x} - \frac y2 \frac {\partial}{\partial z} - \frac{x^2 + y^2}{2} \frac{\partial}{\partial w}, \\
&X_2 = \frac{\partial}{\partial y} + \frac x2 \frac {\partial}{\partial z} + \frac{x^2 + y^2}{2} \frac{\partial}{\partial v}, \\
&X_3 = \frac{\partial}{\partial z} + x \frac{\partial}{\partial v} + y \frac{\partial}{\partial w}, \\
&X_4 = \frac{\partial}{\partial v}, \\
&X_5 = \frac{\partial}{\partial w}.
\end{align*}
The product rule in the Cartan group $M$ in this model is given in \cite{dido_exp}.
Left-invariant $\ell_{\infty}$ sub-Finsler problem on the Cartan group is stated as follows:
\begin{align}
&\dot{q} = u_1 X_1 + u_2 X_2, \quad q \in M, \quad u \in U = \{ u \in \mathbb{R}^2 \mid  {\lVert u \rVert}_\infty \le 1 \}, \label{sys}\\
\nonumber
&\lVert u \rVert_\infty = \max (|u_1|, |u_2|), \\
&q(0) = q_0 = \Id = (0, \dots, 0), \quad q(T) = q_1, \label{bound}\\
&T \to \min. \label{T}
\end{align}
\begin{remark}
Problem \eqref{sys}--\eqref{T} is geometrically stated as the following problem in the plane $\mathbb{R}^2_{xy}$.

Let a point $(x_1, y_1) \in \mathbb{R}^2$, a number $S \in \mathbb{R}$, and a point $c \in \mathbb{R}^2$ be given. The point $(x_1, y_1)$ is connected with the origin by a curve $\gamma_0 \subset \mathbb{R}^2$. One should find a curve $\gamma = \{ (x(t), y(t)) \mid t \in [0, T] \}$, with velocity $\lVert (\dot{x} (t), \dot{y}(t)) \rVert_\infty \le 1$ that connects the origin with the point $(x_1, y_1)$, bounds together with the curve $\gamma_0$ a domain with oriented area $S$ and center of mass $c$, for which the time of motion $T$ is minimal.
\end{remark}

Rashevsky-Chow theorem \cite{notes} implies complete controllability of system \eqref{sys}, and Filippov theorem \cite{notes} implies existence of optimal controls in the time-optimal problem \eqref{sys}--\eqref{T}.
\section{Pontryagin Maximum Principle}

We apply Pontryagin Maximum Principle (PMP) to Problem \eqref{sys}--\eqref{T}. 

Denote points of the cotangent bundle of $M$ as $\lambda \in T^* M$. Introduce linear-on-fibers Hamiltonians $h_i (\lambda) = \langle \lambda, X_i \rangle$, $i = 1, \dots, 5$, and the Hamiltonian of PMP 
$$ h^\nu_u (\lambda) = \langle \lambda, u_1 X_1 + u_2 X_2  \rangle + \nu = u_1 h_1 (\lambda) + u_2 h_2 (\lambda) + \nu, \qquad \lambda \in T^*M, \ u \in U, \ \nu \in \R. $$
 Denote by $\vec{h}_i \in \Vec(T^* M)$ the Hamiltonian vector field corresponding to the Hamiltonian function $h_i$.
 \begin{theorem}
[PMP \cite{PBGM, notes}]\label{PMP}
If a control $u(t)$ and the corresponding trajectory $q(t), t \in [0, T]$, are optimal in Problem~\eq{sys}--\eq{T}, then there exist a curve $\lambda \in \Lip ([0, T], T^* M)$, $\lambda_t \in T_{q(t)} ^* M$, and a number $\nu \le 0$ such that the following conditions hold:
\begin{align}
&\dot{\lambda}_t = u_1(t) \vec{h}_1({\lambda}_t) + u_2(t) \vec{h}_2({\lambda}_t), \label{Hamsys}\\
&u_1 (t) h_1 ({\lambda}_t) + u_2 (t) h_2 ({\lambda}_t) = \max_{v \in U} (v_1 h_1 ({\lambda}_t) + v_2 h_2 ({\lambda}_t)) = H ({\lambda}_t), \label{PMPmax}\\
\nonumber
&H (\lambda) := (|h_1| + |h_2|) (\lambda),\\
&{\lambda}_t \ne 0, \label{lne}\\
&h_{u(t)}^{\nu} ({\lambda}_t) = H({\lambda}_t) +\nu \equiv 0.
\end{align}
\end{theorem}

The following two cases should be distinguished: 
\begin{itemize}
\item[(A)] $\nu = 0 \Leftrightarrow $ extremal $\lambda_t$ is abnormal $ \Leftrightarrow H(\lambda_t) \equiv 0$, 
\item[(N)] $\nu < 0 \Leftrightarrow $ extremal $\lambda_t$ is normal $ \Leftrightarrow H(\lambda_t) \equiv \const  > 0$.
\end{itemize}

The Hamiltonian system of PMP~\eq{Hamsys} reads in coordinates $(h_1, \dots, h_5 ; q)$ as follows:
\begin{align}
&\dot{h}_1 = -u_2 h_3, \label{Ham_sys1}\\
&\dot{h}_2 = u_1 h_3, \label{Ham_sys2}\\
&\dot{h}_3 = u_1 h_4 + u_2 h_5,\label{Ham_sys3}\\
&\dot{h}_4 = \dot{h}_5 = 0, \label{Ham_sys4}\\
&\dot{q} = u_1 X_1 + u_2 X_2. \label{Ham_sys5}
\end{align}

\begin{lemma} \label{lem:Casimir}
The dual of the Lie algebra $L^* = T_{\Id}^*M$ has Casimir functions $h_4$, $h_5$, $E = \frac{h_3^2}{2} + h_1 h_5 - h_2 h_4$.
\end{lemma}
\begin{proof}
$\{ h_4, h_i \} = \{ h_5, h_i \} = \{ E, h_i \} = 0, \quad i = 1, \dots, 5.$
\end{proof}
Thus Hamiltonian system \eqref{Ham_sys1}--\eqref{Ham_sys5} has, in addition to $h_4$ and $h_5$, also integral $E$.
\begin{lemma} \label{lem:ui=+-1}
If there exist $i \in \{ 1, 2 \}$ for which $u_i (t) \equiv 1$ or $u_i (t) \equiv -1$, then the control $u(t)$ is optimal.
\end{lemma}
\begin{proof}
If $u_1 (t) \equiv 1$ (resp. $-1$), then the coordinate $x(t)$ changes with maximum (resp. minimum) possible velocity. Similarly for $u_2 (t)$ and $y(t)$.
\end{proof}

Denote by $\mathcal{A}_{q_0}^{\mathrm{sing}} (T)$ the attainable set of system~\eq{sys} for time $T > 0$ along trajectories starting from point $q_0$ with control $u_i(t) \equiv 1$ or $u_i(t) \equiv -1$ for  $i \in \{1,2\}$.

\begin{definition}
We call a control $u(t)$ and the corresponding trajectory $q(t)$ {\em geometrically optimal}, if $u_i(t) \equiv 1$ or $u_i(t) \equiv -1$ for some $i \in \{1,2\}$ and trajectory $q(t)$ ends at the boundary of set $\mathcal{A}_{q_0}^{\mathrm{sing}} (T)$, i.e., $q(T) \in \partial \mathcal{A}_{q_0}^{\mathrm{sing}} (T).$
\end{definition}

In order to describe the boundary of $\mathcal{A}_{q_0}^{\mathrm{sing}} (T)$, we apply the geometric formulation of PMP~\cite{notes}. It has the following formulation for Cauchy problem~\eq{sys}--\eq{bound} with condition $u_2(t) \equiv 1$, similar to the formulation of Theorem~\ref{PMP}.   
	
\begin{theorem}[Geometric formulation of PMP~\cite{notes}]\label{PMPgeom}
If a control $u(t)$ with condition $u_2(t) \equiv 1$ and the corresponding trajectory $q(t), t\in[0,T]$, are geometrically optimal in problem~\eq{sys}--\eq{bound}, then there exists a curve in cotangent bundle $\lambda_t \in T_{q(t)}^* M$, such that conditions~\eq{Hamsys}, \eq{lne} and the following maximum condition hold:
\begin{align} \label{maxGeom}
 u_1 (t) h_1 ({\lambda}_t) + h_2 ({\lambda}_t) = \max_{\bar{u}_1 \in [-1,1]} (\bar{u}_1 h_1 ({\lambda}_t) + h_2 ({\lambda}_t)) = H ({\lambda}_t).
\end{align}
\end{theorem}
\begin{corollary}\label{CorGeom}
Hamiltonian system for geometric formulation of PMP in coordinates $(h_1, \dots, h_5;q)$ coincides with system $(\ref{Ham_sys1})$--$(\ref{Ham_sys5})$ provided $u_2 \equiv 1$.
\end{corollary}

\begin{lemma} \label{lem:|u|=1}
If a control $u(t), t \in [0, T]$, is optimal, then $\lVert u(t) \rVert_\infty \equiv 1$ for a.e. $t \in [0, T]$.
\end{lemma}
\begin{proof}
If $\lVert u(t) \rVert_\infty < 1$ on a subset of $[0, T]$ of positive measure, then the trajectory
$q(t)$, $t \in [0, T]$, can be reparametrized and be passed by a time less than $T$.
\end{proof}

\section{Abnormal trajectories}
Let $\nu = 0$.
\begin{theorem} \label{th:abnorm}
Optimal abnormal controls have the form
\be{abnorm_opt}
 u(t) \equiv \const, \quad \lVert u(t) \rVert_\infty \equiv 1,
\ee 
and all such controls are optimal.

These controls define optimal synthesis on the abnormal manifold of the distribution $\spann(X_1, X_2)$:
$$ A = \{ e^{(u_1 X_1 + u_2 X_2)} (\Id) \mid u_i \in \mathbb{R} \} = \{ q \in M \mid z = 0, \ v = y(x^2 +  y^2)/6, \ w = - x(x^2 +  y^2)/6 \}. $$

Abnormal trajectories are one-parameter subgroups of $M$ tangent to the distribution $\spann(X_1, X_2)$.
Projections of abnormal trajectories to the plane $(x,y)$ are straight lines.
\end{theorem}
\begin{proof}
In the abnormal case we have $\nu = - H(\lambda_t) \equiv 0$. It follows from the maximality condition~\eq{PMPmax}
of PMP that $h_1 (\lambda_t) = h_2 (\lambda_t) \equiv 0$ along an abnormal extremal. For an optimal extremal $(u_1^2 + u_2^2) (t) \ne 0$, then equations~\eq{Ham_sys1}, \eq{Ham_sys2} yield the identity $h_3 (\lambda_t) \equiv 0$, whence equation \eqref{Ham_sys3} gives $u_1 (t) h_4 (\lambda_t) + u_2 (t) h_2 (\lambda_t) \equiv 0$.

Summing up, optimal abnormal extremals satisfy the conditions: $h_1 (\lambda_t) = h_2 (\lambda_t) = h_3 (\lambda_t) \equiv 0$, $\dot{h}_4 (\lambda_t) = \dot{h}_5 (\lambda_t) \equiv 0$, $u(t) \perp (h_4, h_5) (\lambda_t)$. If an abnormal control $u(t)$ is nonconstant, then it is not optimal. Lemma \ref{lem:|u|=1} implies that $\lVert u(t) \rVert_\infty \equiv 1$.

Optimality of all controls \eqref{abnorm_opt} follows from Lemma \ref{lem:ui=+-1}.

Optimal abnormal trajectories are one-parameter subgroups of the Cartan group tangent to the distribution $\spann(X_1, X_2)$, their parameterization is as follows:
\begin{gather}
\nonumber
x = u_1 t, \quad  y = u_2 t, \quad z = 0,  \quad 
v = u_2 \frac{u_1^2 + u_2^2}{6} t^3,
 \quad 
w = - u_1 \frac{u_1^2 + u_2^2}{6} t^3.
\end{gather}
\end{proof}

\section{Types of normal extremal arcs}\label{sec:types}
Now we start to consider normal extremals, i.e., we assume that $-\nu = H(\lambda_t) > 0$.

We call a normal extremal arc $\lambda_t, t \in I = (\alpha, \beta) \subset [0, T]$:
\begin{itemize}
\item a bang-bang arc if
$$ \card \{ t \in I \mid h_1 h_2 (\lambda_t) = 0 \} < \infty, $$
\item a singular arc if one of the condition holds:
\begin{align*}
&h_1 (\lambda_t) \equiv 0, \quad t \in I \quad (\textrm{$h_1$-singular arc}), \text{ or}\\
&h_2 (\lambda_t) \equiv 0, \quad t \in I \quad (\textrm{$h_2$-singular arc}),
\end{align*}
\item a mixed arc if it consists of a finite number of bang-bang and singular arcs.
\end{itemize}

Notice that a priori this list of possible types of normal arcs is not complete: e.g. there
could happen Fuller phenomenon. But a posteriori we will prove  that this list is in fact complete.

\begin{remark}
If $h_i (\lambda_t) |_{(\alpha, \beta)} \ne 0$, then $u_i (t) |_{(\alpha, \beta)} \equiv s_i := \sgn h_i (\lambda_t)|_{(\alpha, \beta)}$. Thus a bang-bang control is piecewise constant, with values in the vertices $(\pm 1, \pm 1)$ of the square $U$.
\end{remark}

\section{Singular arcs}

\subsection{Characterization of singular arcs}
\begin{theorem} \label{th:h1-sing}
Each $h_1$-singular normal arc satisfies one of the following conditions:
\begin{itemize}
\item[$(a)$] $h_1 = h_3 = h_4 = h_5 \equiv 0, \quad h_2 \equiv \const \ne 0,$\\
$\ |u_1 (t)| \le 1$, $u_2 \equiv s_2 \in \{ \pm 1 \}$,
\item[$(b)$]
 $h_1 = h_3 \equiv 0$, $|\frac{h_5}{h_4}| \le 1$, $h_4 \ne 0$, $h_2 \equiv \const \ne 0$,\\
$ u_1 (t) \equiv u_1^a = -s_2 \frac{h_5}{h_4}$, $u_2 (t) \equiv s_2 \in \{ \pm 1 \}$.
\end{itemize}
\end{theorem}
\begin{proof}
By definition, along an $h_1$-singular arc we have: 
$$ h_1 (\lambda_t) \equiv 0, \quad H(\lambda_t)\ne 0 \quad \Rightarrow \quad h_2 (\lambda_t) \ne 0, \quad t \in I = (\alpha, \beta) \subset [0, T]. $$
By the maximality condition of PMP, $ u_2 (t) |_I \equiv s_2 \in \{ \pm 1 \}$. From the Hamiltonian system of PMP, $ h_3 (\lambda_t) = - \frac{\dot{h}_1 (\lambda_t)}{u_2 (t)} \equiv 0$, thus $h_2 (\lambda_t) \equiv \const \ne 0.$ Moreover, $u_1 (t) h_4 + s_2 h_5 = \dot{h}_3 (\lambda_t) \equiv 0.$

$(a)$ 
If $h_4 = 0$, then $h_5 = 0$, and we get statement $(a)$ of this theorem.

$(b)$
If $h_4 \ne 0$, then $u_1 (t) = u_1^a = -s_2 \frac{h_5}{h_4}$, and statement $(b)$ of this theorem follows.
\end{proof}
A similar statement for $h_2$-singular arcs holds.
\begin{theorem} \label{th:h2-sing}
Each $h_2$-singular normal arc satisfies one of the following conditions:
\begin{itemize}
\item[$(a)$]
 $h_2 = h_3 = h_4 = h_5 \equiv 0$, $h_1 \equiv \const \ne 0$, \\$u_1 \equiv s_1 \in \{ \pm 1 \}$,
$|u_2 (t)| \le 1,$
\item[$(b)$]
 $h_2 = h_3 \equiv 0$, $h_1 \equiv \const \ne 0$, $h_5 \ne 0$, $|\frac{h_4}{h_5}| \le 1$, \\
$u_1 \equiv s_1 \in \{ \pm 1 \}$, $u_2 \equiv u_2^a=-s_1 \frac{h_4}{h_5}$.
\end{itemize}
\end{theorem}
\begin{corollary}
All singular trajectories are optimal.
\end{corollary}
\begin{proof}
Follows from Lemma \ref{lem:ui=+-1}.
\end{proof}
\begin{remark}
Singular trajectories of type $(b)$ are simultaneously normal and abnormal.
\end{remark}

\subsection{Attainable set via singular trajectories}
In this subsection we describe the attainable set $\mathcal{A}_{q_0}^{\mathrm{sing}} (T)$ for system~(\ref{sys}) via singular trajectories.

Notice that all four cases $u_1 \equiv 1$, $u_1 \equiv -1$, $u_2 \equiv 1$, $u_2 \equiv -1$ are symmetric, therefore it is enough to study just one of them.

We consider the case $u_2 \equiv 1$. We get the following vertical subsystem of Hamiltonian system from Corollary~\ref{CorGeom}:
\begin{equation} \label{eq:SingVerth}
\begin{cases}	
\dot{h}_1 = - h_3, \\ 
\dot{h}_2 = u_1 h_3, \\
\dot{h}_3 = u_1 h_4  + h_5, \\ \dot{h}_4 = \dot{h}_5 = 0. 
\end{cases}
\end{equation}



If $h_4 = 0$, then $h_1(t) = -\frac{h_5}{2} t^2  - h_3(0)t+h_1(0)$. It follows from maximality condition~\eq{maxGeom} that the corresponding control $u_1$ is piecewise constant and has up to two switchings in the set $\{-1,1\}$.

Further, consider the case $h_4 \neq 0$. Via the dilation symmetry along the parameter $|h_4|$, we pass to new coordinates:
\begin{align*}
\hf_i = \frac{h_i}{|h_4|}, \qquad i=1,\dots,5.
\end{align*} 

The vertical subsystem takes the form
\begin{equation} \label{eq:SingVerthf}
\begin{cases}
\dot{\hf}_1 = - \hf_3, \\ \dot{\hf}_2 = u_1 \hf_3, \\  \dot{\hf}_3 = \hf_4 u_1 + \hf_5, \\ \dot{\hf}_5 = 0, \qquad \hf_4 = \pm 1. 
\end{cases}
\end{equation}
If $\hf_1 = 0, t \in(t_0,t_1), t_0 < t_1$, then we get $\hf_3=0$ from~(\ref{eq:SingVerthf}), therefore $u_1 = - \frac{\hf_5}{\hf_4} = u_1^a$ for $t \in(t_0,t_1), t_0 < t_1$. This gives us the following case:
\begin{align}\label{singabnorm}
\hf_1=\hf_3=0, \quad \hf_4 = \pm 1, \quad \hf_5 \in [-1,1] \qquad \Rightarrow \qquad u_1 = - \frac{\hf_5}{\hf_4}.    
\end{align}

We have $u_1 = \sgn {\hf}_1$ from    maximality condition~\eq{maxGeom} for $\hf_1 \neq 0, t\in (t_0,t_1), t_0 < t_1$, and also $h_2 = H - |h_1|$, therefore we can exclude from the vertical subsystem the equation for $\hf_2$. We get
\begin{equation} \label{eq:SingVerthfu}
\begin{cases}
\dot{\hf}_1 = - \hf_3, \\ \dot{\hf}_3 = \hf_4 \sgn \hf_1 + \hf_5, \\ \dot{\hf_5} = 0, \qquad \hf_4 = \pm 1. 
\end{cases}
\end{equation}

Using the symmetry $(\hf_1,\hf_3,\hf_4,\hf_5) \mapsto (-\hf_1,-\hf_3,\hf_4,-\hf_5)$, it is enough to consider the subcases $\hf_4 = \pm 1$ with the condition $\hf_5\geq 0$. There exists a natural decomposition for the set of adjoint vectors:
\begin{align*}
\Cf^{\pm} &= \{\hf=(\hf_1,\hf_3,\hf_4,\hf_5)\in \R^4 \mid (\hf_1,\hf_3) \in \R^2, \hf_4 = \pm 1, \hf_5 \geq 0\} = \Cf_0^{\pm} \cup \Cf_{01}^{\pm} \cup \Cf_1^{\pm} \cup \Cf_{1\infty}^{\pm}, \\
\Cf_0^{\pm} &= \{\hf\in \Cf^{\pm} \mid \hf_5 = 0\}, \qquad \Cf_{01}^{\pm} = \{\hf\in \Cf^{\pm} \mid \hf_5 \in (0, 1)\}, \\
\Cf_1^{\pm} &= \{\hf\in \Cf^{\pm} \mid \hf_5 = 1\}, \qquad \Cf_{1\infty}^{\pm} = \{\hf\in \Cf^{\pm} \mid \hf_5 \in (1, +\infty)\}.
\end{align*} 

Each subset is studied with the help of phase portrait of system~(\ref{eq:SingVerthfu}) in the plane $(\hf_1,\hf_3)$. Along all trajectories of vertical subsystem~(\ref{eq:SingVerthfu}), we follow over the sign of $\hf_1$, which defines the desired control $u_1$. Subsystem~\eq{eq:SingVerthfu} has an obvious first integral $\frac{1}{2} \hf_3^2 + \hf_5 \hf_1 + \hf_4 |\hf_1| = E + h_4 H$. 

\begin{figure}[h]
\centering
\includegraphics[width=0.35\linewidth]{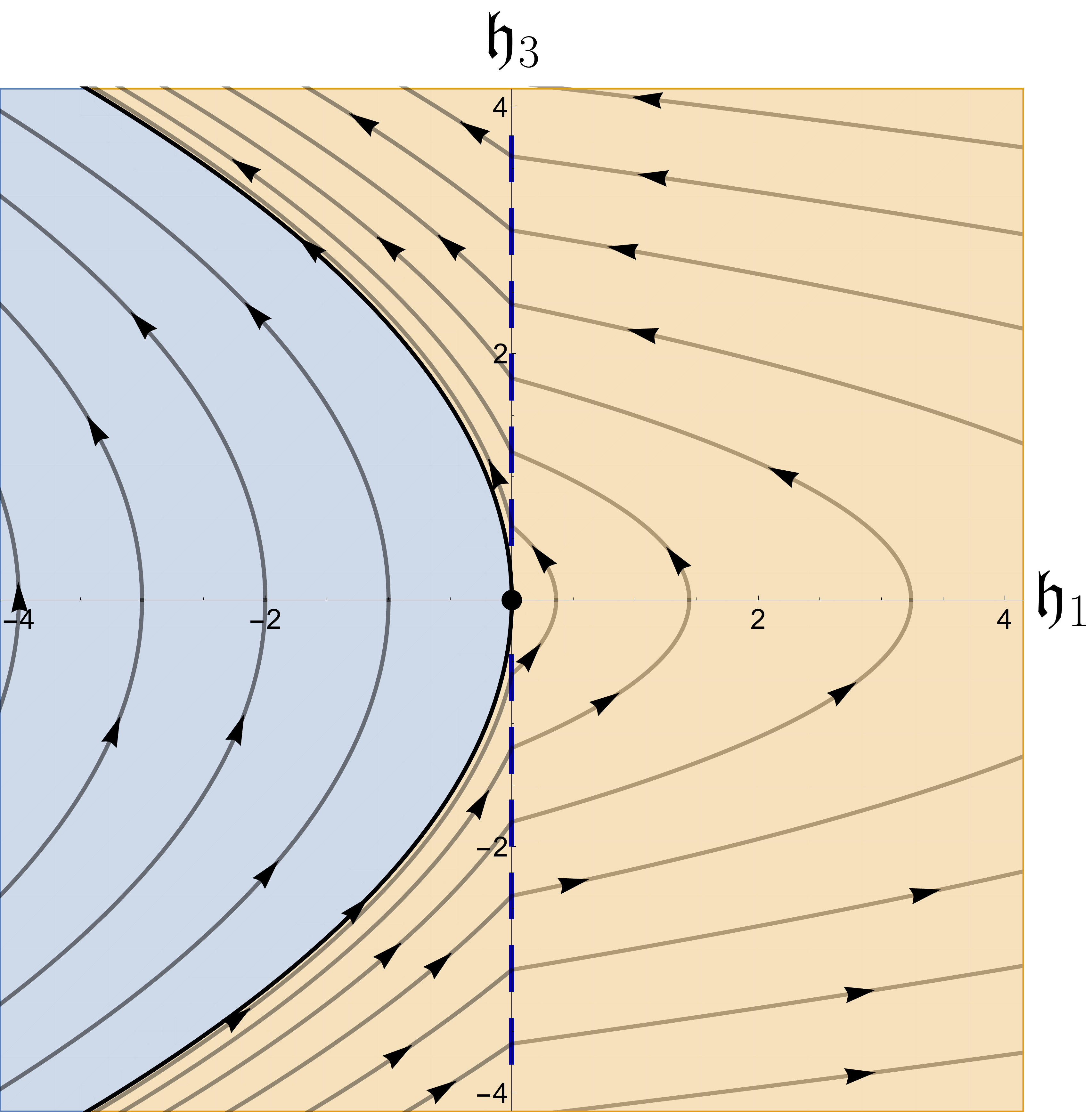}
\qquad\qquad
\includegraphics[width=0.35\linewidth]{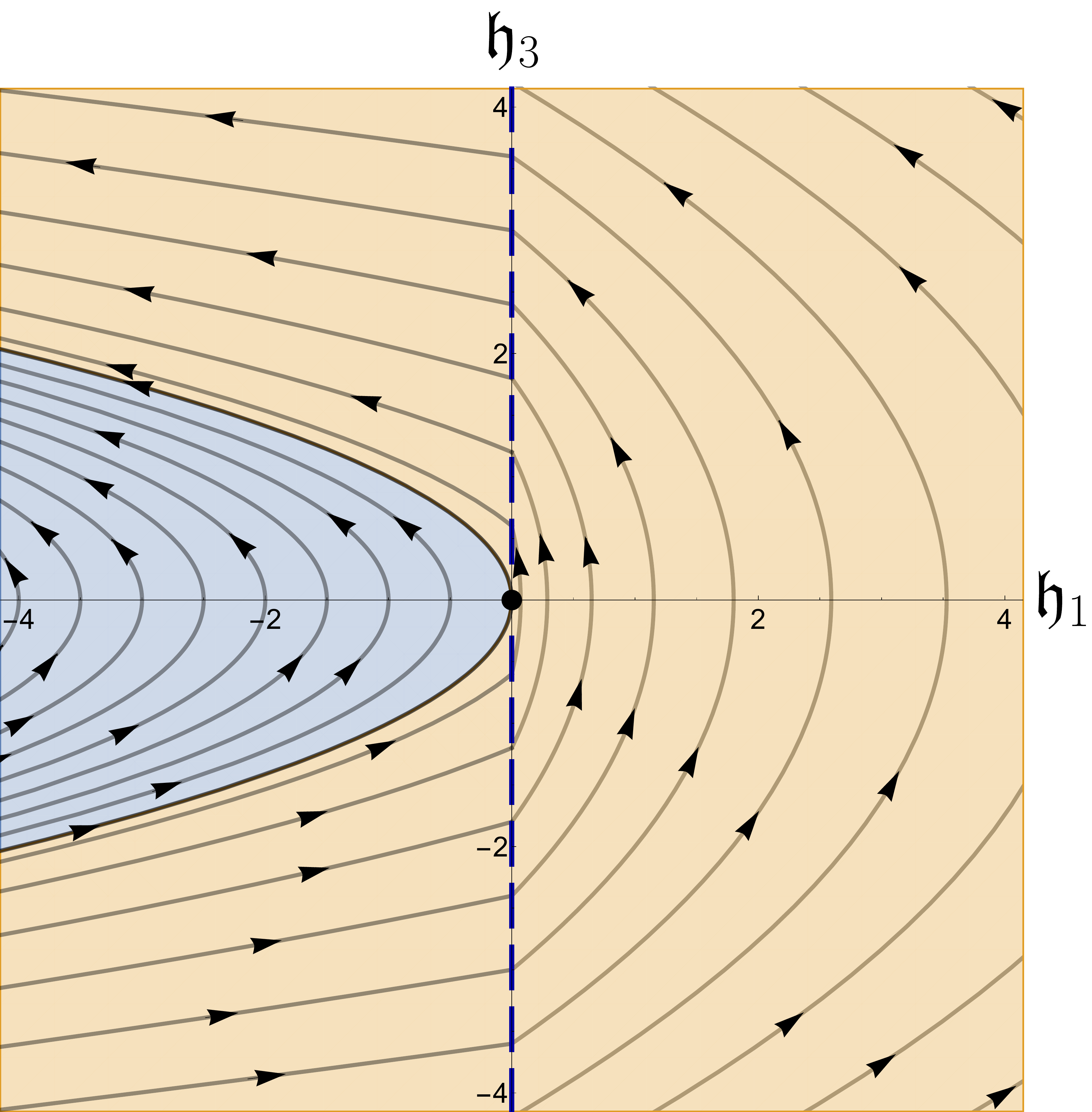}
\caption{
Phase portrait of vertical subsystem~(\ref{eq:SingVerthfu}) for $\hf_5 = 3/2$; (left) $\hf_4=-1$, (right) $\hf_4=1$.
}\label{cf1i}
\end{figure}

The phase portraits for $\hf_5 = 3/2$ with $\hf_4=-1,1$ are given in Fig.~\ref{cf1i}. The both cases $\hf_4=-1,1$ give us parabolas which open to the left. Variation of the parameter $\hf_5$ in the set $\Cf_{1\infty}^{\pm}$ does not change the direction of opening of parabolas. Notice that it is impossible to switch to control~(\ref{singabnorm}) for $\hf\in\Cf_{1\infty}^{\pm}$ since the control $u_1 = \mp \hf_5 \in (-\infty,-1)\cup(1,+\infty)$ is not admissible. Therefore, the corresponding controls are piecewise constant and have up to two switchings in the set $\{-1,1\}$.

\begin{figure}[h]
\centering
\includegraphics[width=0.35\linewidth]{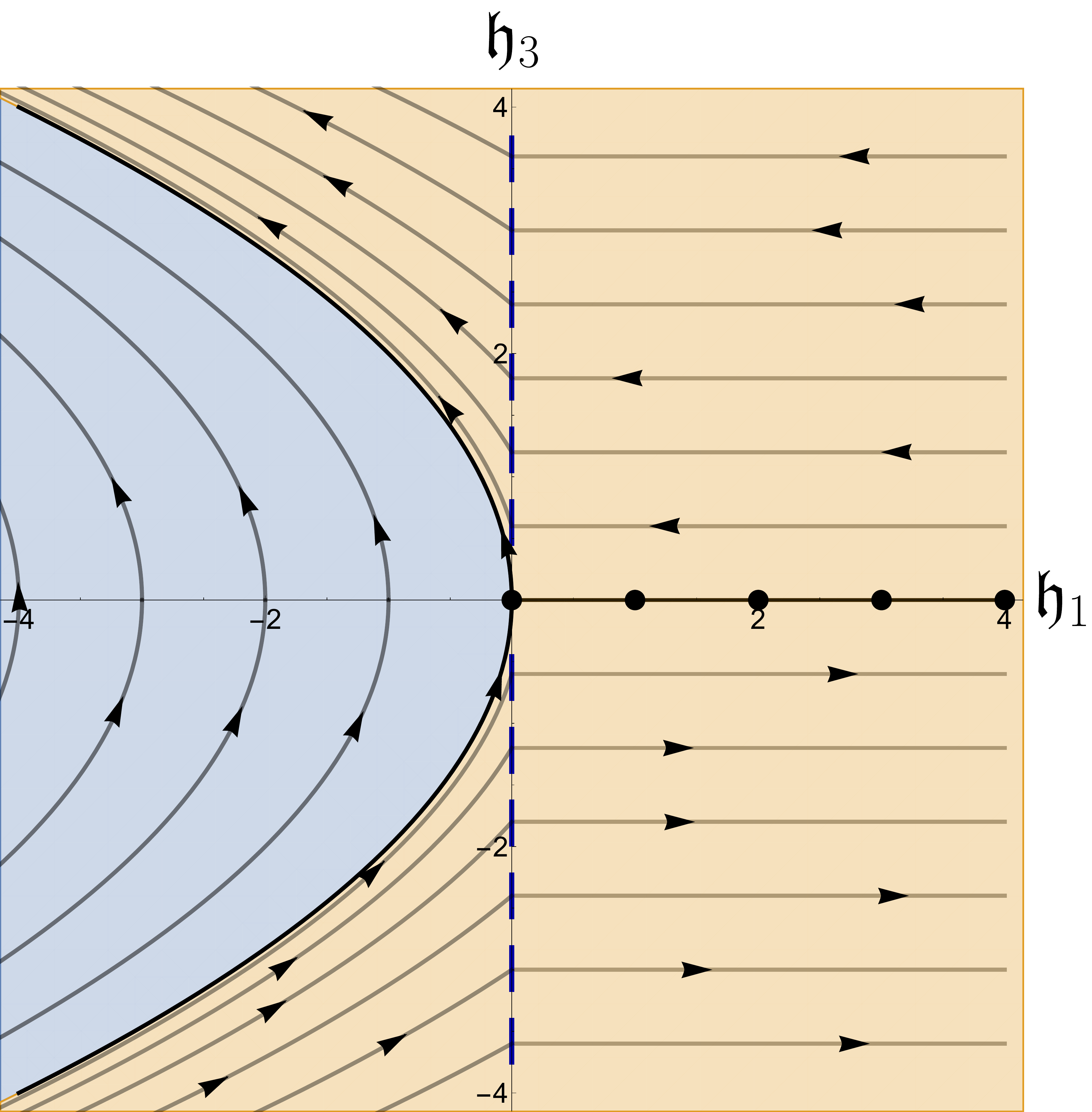}
\qquad\qquad
\includegraphics[width=0.35\linewidth]{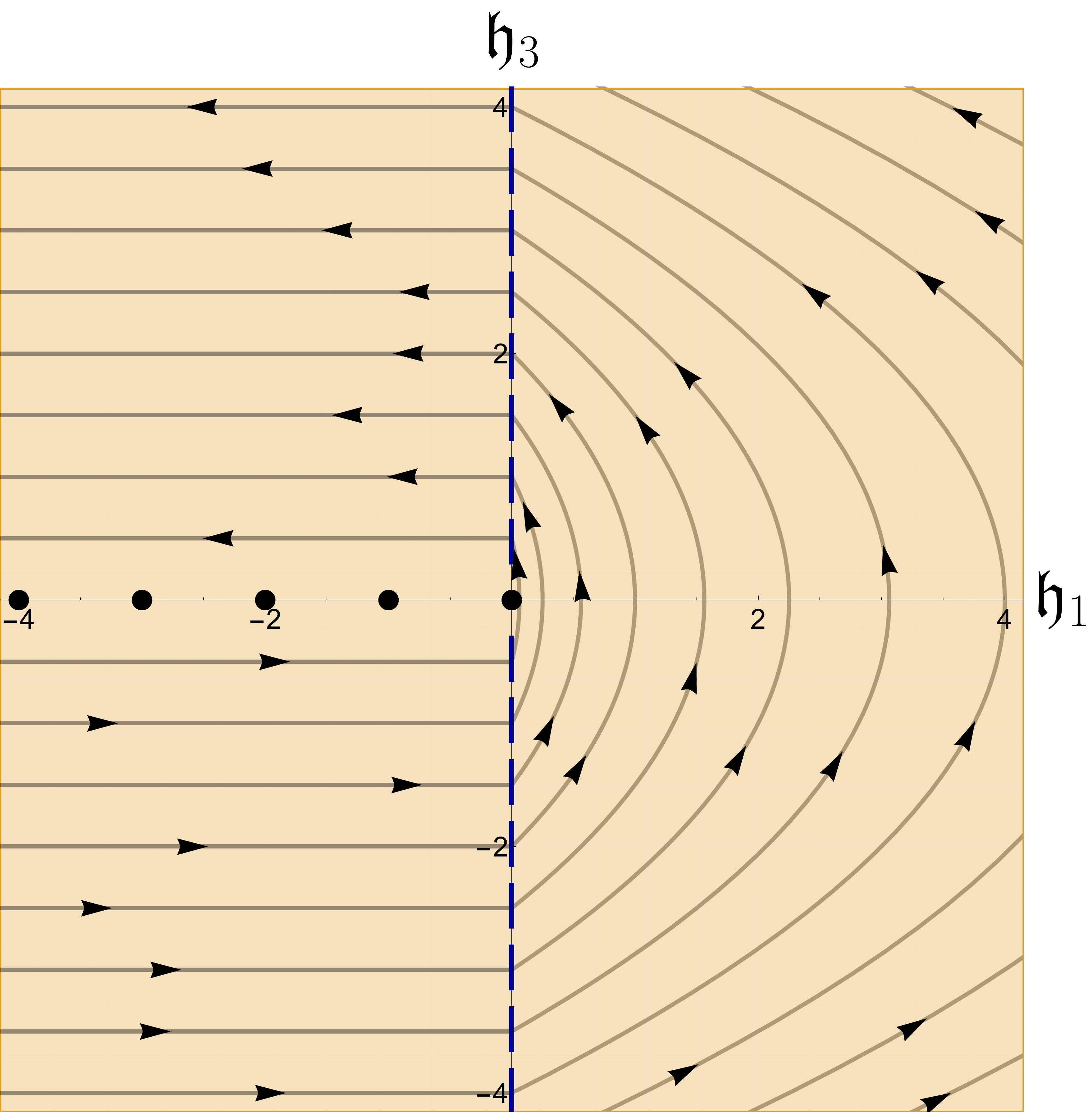}
\caption{
Phase portraits of vertical subsystem~(\ref{eq:SingVerthfu}) for $\hf_5 = 1$; (left) $\hf_4=-1$, (right) $\hf_4=1$.
}\label{cf1}
\end{figure}

Fig.~\ref{cf1} shows the phase portraits with the condition $\hf_5 = 1$. There exists a parabola for $\hf_4=-1$ (left Fig.~\ref{cf1}) which goes through the origin, where we can switch to control~(\ref{singabnorm}); this parabola corresponds to a piecewise constant control $u_1$, which has up to two switchings in the set  $\{-1,1\}$. The control for the other trajectories from the set $\Cf_1^-$ is piecewise constant as well and has not more than one switching in the set $\{-1,1\}$. The origin in subcase $\hf_4=1$ is a stable equilibrium, therefore switching to control~(\ref{singabnorm}) cannot happen along trajectories from the set $\Cf_1^+$. Here we have also piecewise constant control $u_1$ with up to two switchings in the set $\{-1,1\}$.

\begin{figure}[h]
\centering
\includegraphics[width=0.35\linewidth]{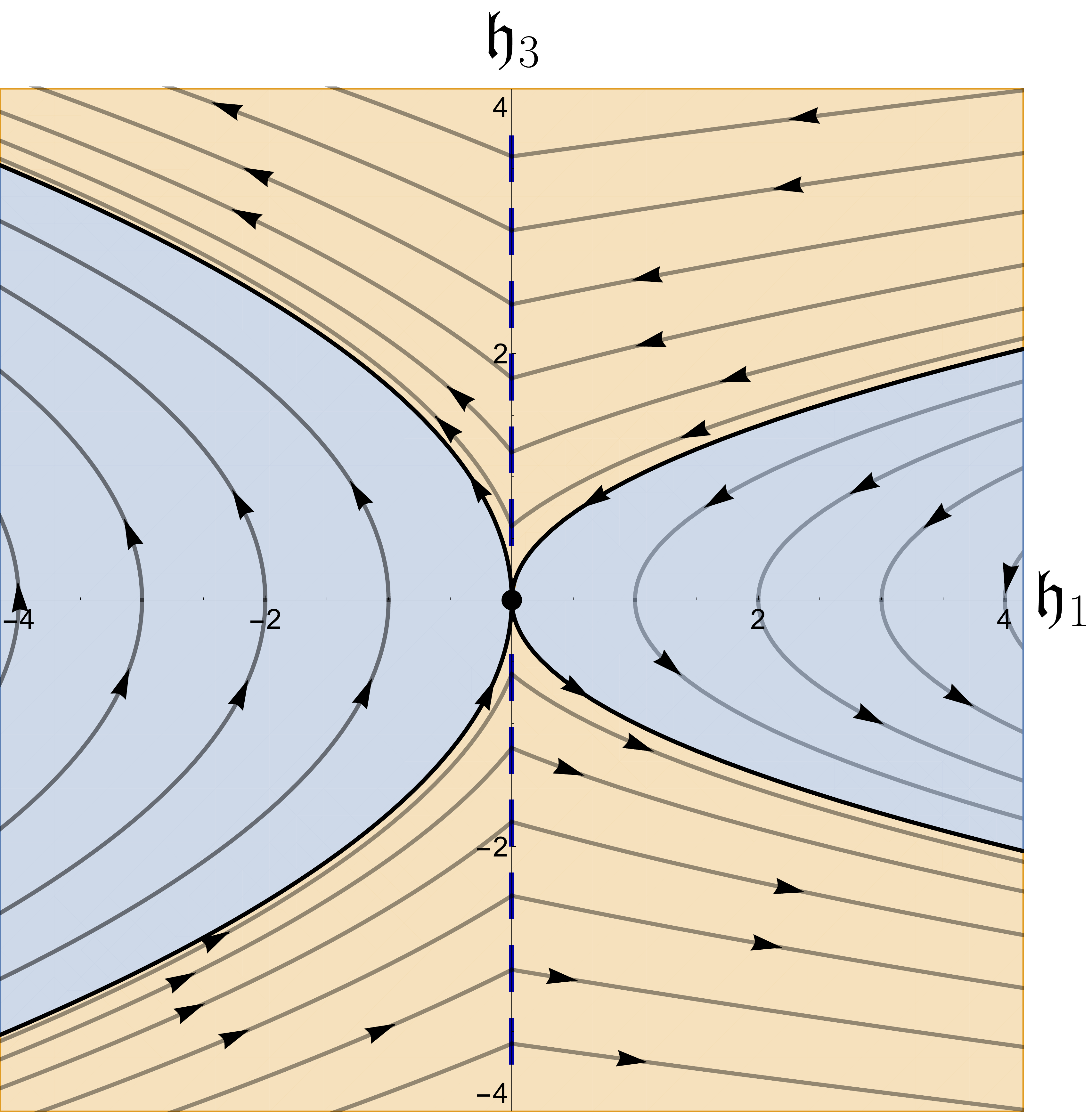}
\qquad\qquad
\includegraphics[width=0.35\linewidth]{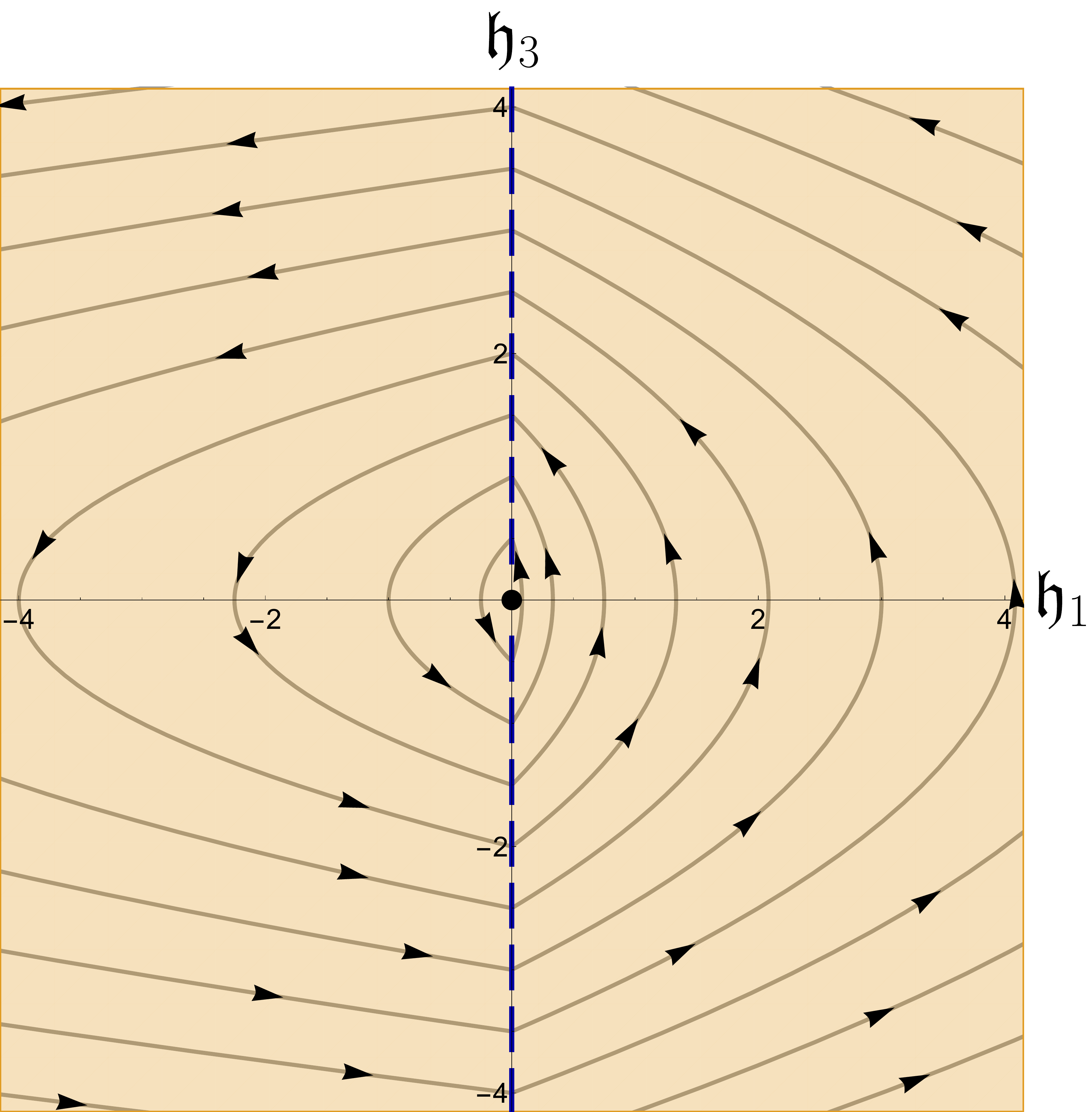}
\caption{
Phase portraits of vertical subsystem~(\ref{eq:SingVerthfu}) for $\hf_5 = 1/2$; (left) $\hf_4=-1$, (right) $\hf_4=1$.
}\label{cf01}
\end{figure}

Fig.~\ref{cf01} shows phase portraits with the condition $\hf_5 = 1/2$. The both cases $\hf_4=-1,1$ give us parabolas which are open in opposite directions in the half-planes. It is impossible to change the direction of opening of the parabolas by varying parameter $\hf_5$ in the set  $\Cf_{01}^{\pm}$. There exist two parabolas for $\hf_4=-1$ (in the left and right half-plane), which go through the origin, where we can switch to control~(\ref{singabnorm}), this case provides us piecewise constant control $u_1$ with values $\pm 1, \hf_5, \pm 1$. The control $u_1$ for the other trajectories from the set $\Cf_{01}^-$ is piecewise constant and has not more than one switching in the set $\{-1,1\}$. The origin is a stable equilibrium for $\hf_4=1$, therefore switching to control~(\ref{singabnorm}) is not possible here, the case $\Cf_{01}^+$ allows unlimited number of switchings in the set $\{-1,1\}$.  

The case $\hf \in \Cf_0^{\pm}$ is limit for $\hf \in \Cf_{01}^{\pm}$ and does not require separate consideration, it has symmetric parabolas in the left and right half-planes.

We have proved the following theorem about extremal controls.
\begin{theorem}\label{thmsing}
$h_1$-singular trajectories with $u_2 \equiv 1$, endpoints of which form a set containing the boundary of the attainable set $\mathcal{A}_{q_0}^{\mathrm{sing}} (T)$, have one of the following types of piecewise constant controls $u_1$:
\begin{enumerate}
\item  with the values $\ \pm 1, \hf_5, \pm 1 \ $ or $\  \pm 1, \hf_5, \mp 1 \ $ and no limitations on the time periods, where $\hf_5 \in [-1,1]$; 
\item  with the values $\ \pm 1, \mp 1, \pm 1, \mp 1, \dots,\ $ and the time periods $\ T_b, T_1, T_2, T_1, \dots, T_i, T_e$, \  s.t. \ $0<T_b\leq T_2$, \ $0<T_1$, \ 0 $\leq T_e\leq T_{3-i}$, where $i=1,2$. 
\end{enumerate}
\end{theorem}


Now we find an upper bound for the number of switchings of control of the second type, it allows us to reduce geometrically non-optimal controls. 

\begin{theorem}\label{thmsing2}
Let $u_1$ be a piecewise constant control with the values $\, 1, - 1, 1, - 1\, $ and the time periods $\, T_b, T_1, T_2, T_e$, s.t.~$0<T_b\leq T_2, \ 0<T_1, \ 0 \leq T_e\leq T_1$. If $T_e>\frac{T_2-T_b}{T_2+T_b} T_1$, then the control $u_1$ is not geometrically optimal.
\end{theorem} 
\begin{proof}
It is easy to check that the corresponding trajectory $q$ has vanishing coordinate $z$ when $T_e = \frac{T_2-T_b}{T_2+T_b} T_1$. Moreover, there exists a symmetric trajectory corresponding to the piecewise constant control with the values $-1, 1, -1, 1$ and the time periods $T_e, T_2, T_1, T_b$, it comes to the same point with vanishing coordinate $z$.  

We prove the theorem by contradiction.

Let $T_b<T_2$. Suppose that the control $u_1$ is geometrically optimal for $T_e>\frac{T_2-T_b}{T_2+T_b} T_1$. Then the symmetric control with the values $-1, 1, -1, 1, -1$ and the time periods $\frac{T_2-T_b}{T_2+T_b} T_1, T_2, T_1, T_b, T_e-\frac{T_2-T_b}{T_2+T_b}T_1$ is geometrically optimal as well. Since $T_b \neq T_2$, we get a contradiction --- the obtained symmetric control does not belong to any type  of extremal controls (see Theorem~\ref{thmsing}), therefore it cannot be geometrically optimal. 

Let $T_b=T_2$. Suppose there exists $T_e>0$, s.t. the control $u_1$ is geometrically optimal. Then the control with the times $\frac{T_1-T_e/2}{T_1+T_e/2}T_2,T_1,T_2,T_e$ is also geometrically optimal, since it is contained in the initial one. By the proof given above for $T_b<T_2$ we have the following condition:
\begin{align*}
T_e \leq \frac{T_2-\frac{T_1-T_e/2}{T_1+T_e/2}T_2}{T_2+\frac{T_1-T_e}{T_1+T_e}T_2} T_1 = \frac{T_1+T_e/2-(T_1-T_e/2)} {T_1+T_e/2+(T_1-T_e/2)} T_1 = \frac{T_e}{2}.
\end{align*} 
This contradiction completes the proof of the theorem.
\end{proof}

\begin{remark}
Numerical simulation shows that the time $T_e=\frac{T_2-T_b}{T_2+T_b} T_1$ is the cut time, i.e., the corresponding controls are geometrically optimal for $T_e\leq\frac{T_2-T_b}{T_2+T_b} T_1$.
\end{remark}

\begin{theorem}\label{thmsing3}
$h_1$-singular trajectories with $u_2 \equiv 1$, the endpoints of which form the boundary of the attainable set $\mathcal{A}_{q_0}^{\mathrm{sing}} (T)$, have one of the following types of piecewise constant control $u_1$:
\begin{enumerate}
\item with values $\, \pm 1, \hf_5, \pm 1\, $ or $\, \pm 1, \hf_5, \mp 1\, $ and no limits on time periods, where $\hf_5 \in [-1,1]$ (up to $2$ switchings); 
\item with values $\, \pm 1, \mp 1, \pm 1, \mp 1\, $ and time periods $\, T_b,T_1, T_2,T_e$, s.t. $0<T_b< T_2, \ 0<T_e \leq \frac{T_2-T_b}{T_2+T_b} T_1$ ($3$ switchings). 
\end{enumerate}
\end{theorem}
\begin{proof}
Follows from Theorems~\ref{thmsing},\ref{thmsing2}.
\end{proof}

\begin{prop}
For any $q_1 \in \mathcal{A}_{q_0}^{\mathrm{sing}} (T)$ there exists a piecewise constant singular control $\bar{u}_1$ with not more than $4$ switchings, s.t. the corresponding trajectory $\bar{q}(t) = q_{\bar{u}}(t)$, $t \in [0,T],$ comes to the point $q_1=\bar{q}(T)$. 
\end{prop}
\begin{proof}
We prove the theorem for case $q_1 = (x_1,T,z_1,v_1,w_1)$, other cases can be proved similarly.

If $q_1 \in \partial \mathcal{A}_{q_0}^{\mathrm{sing}} (T)$ then by Theorem~\ref{thmsing3} there exists a piecewise constant singular control $\bar{u}_1(t)$ , $t\in[0,T]$, with not more than $3$ switchings which leads to the point $q_1$.

Otherwise, we consider a constant control $u_1^0 \equiv 0$ and the corresponding singular trajectory $q^0(s) = q_{u^0}(t)$, $s \in [0,T]$. By continuity there exists a time $s^0$, s.t. $q_1 \in \partial \mathcal{A}_{q^0(s^0)}^{\mathrm{sing}} (T-s^0)$. From Theorem~\ref{thmsing3} there exists a piecewise constant singular control $\bar{u}_1(t)$, $t\in [s^0,T]$ with not more than $3$ switchings which connects the point $q^0(s^0)$ with  point $q_1$. Define $\bar{u}_1(t)=0$, $t\in [0,s^0]$, thus we construct a control $\bar{u}_1$ with not more than $4$ switchings which connects $q_0=\bar{q}(0)$ with $q_1=\bar{q}(T)$ by a singular trajectory $\bar{q}(t) = q_{\bar{u}_1}(t)$, $t\in [0,T]$.  
\end{proof}

\begin{figure}[h]
\centering
\includegraphics[width=0.45\linewidth]{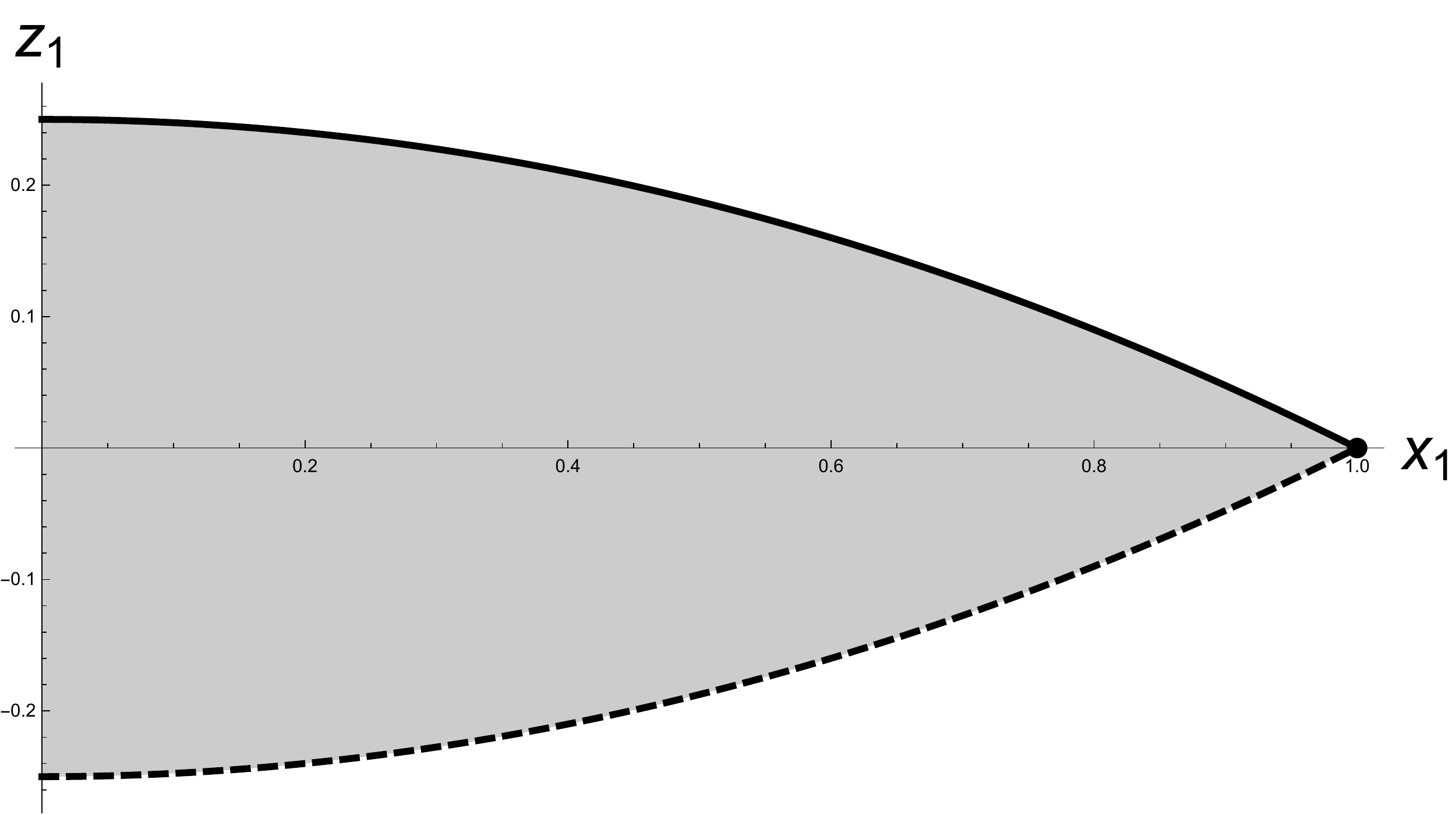} \quad \includegraphics[width=0.45\linewidth]{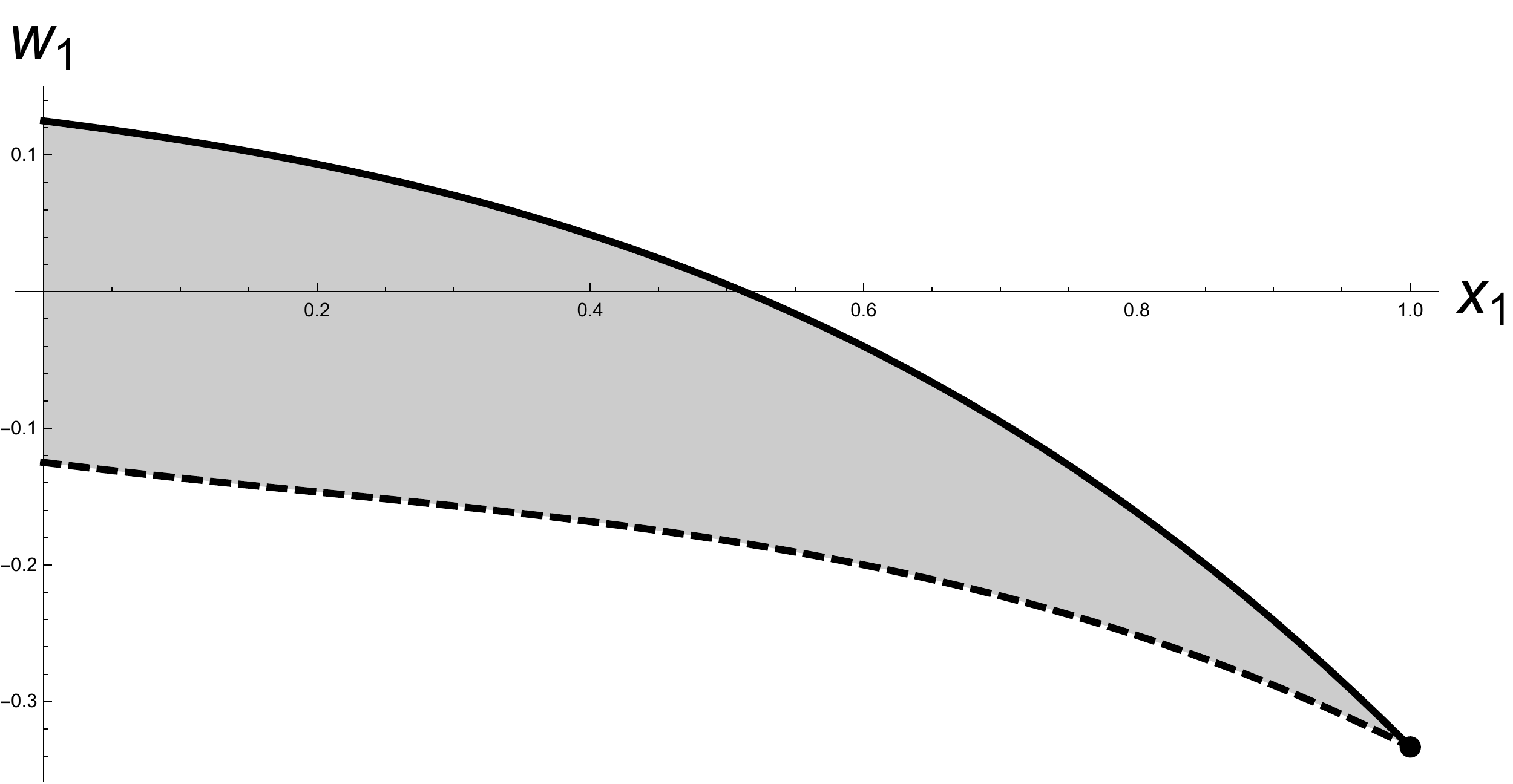}
\caption{
Projections of the attainable set to the planes $(x_1,z_1)$ and $(x_1,w_1)$. 
}\label{at-xzw}
\includegraphics[width=0.5\linewidth]{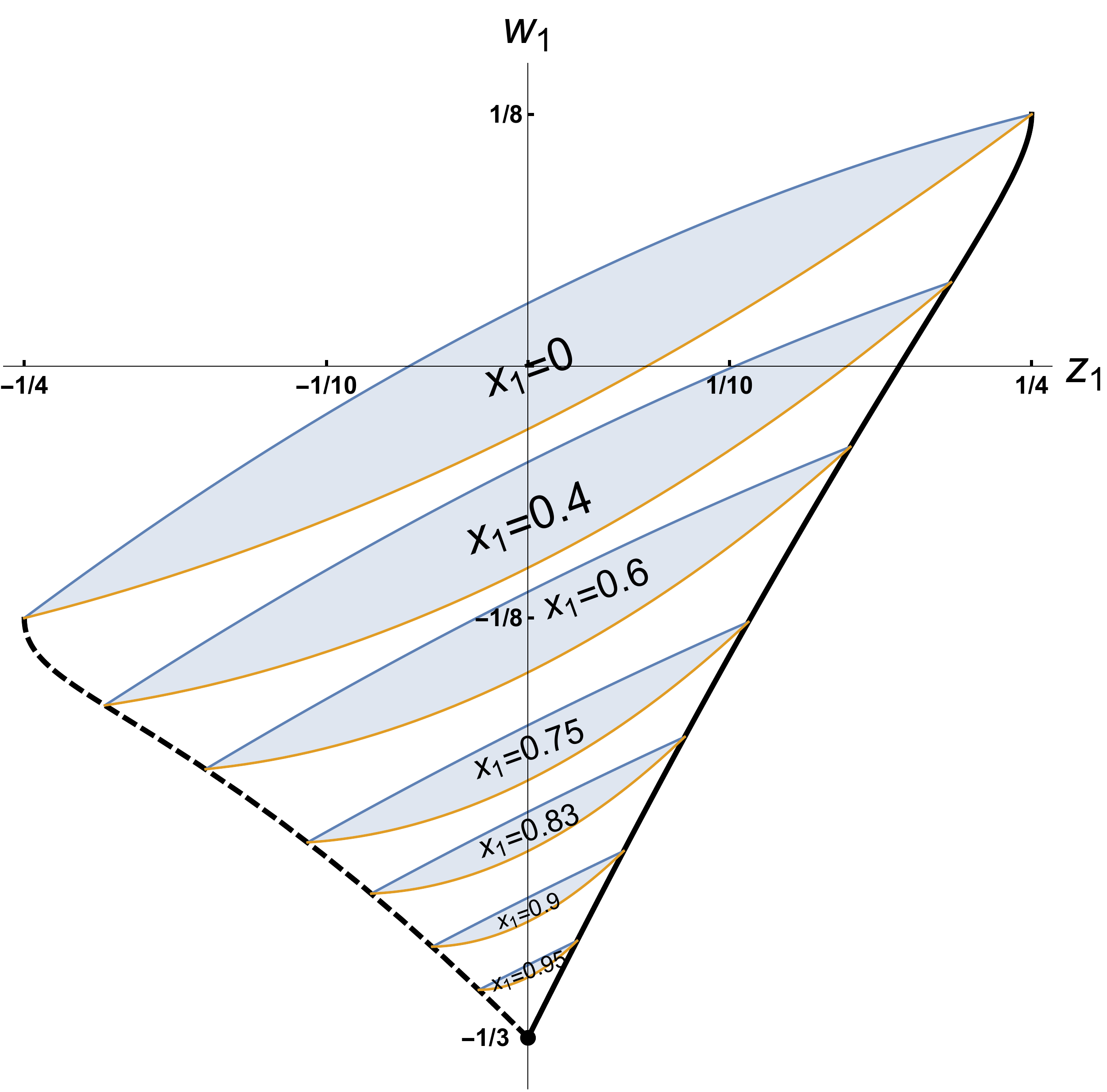} 
\caption{
Projections to the plane $(z_1,w_1)$ of sections of the attainable set with fixed values $x_1$. 
}\label{at-zw}
\end{figure}

Using controls from Theorem~\ref{thmsing3}, we construct the attainable set $\mathcal{A}_{q_0}^{\mathrm{sing}} (T)$ for a terminal time $T$ in the hyperplane $(x_1,z_1,w_1,v_1)$ with $y_1=T$. Let $T=1$, then there is an obvious restriction for the first coordinate $|x_1|\leq 1$. Using the reflection symmetry generated by changing sign of control $u_1$, it is enough to consider the case 
\begin{align}
0\leq x_1\leq 1. \label{restrx}
\end{align}
Notice that a solution for $x_1=1$ is unique and is given by the trajectory with $u_1 \equiv 1$ which comes to the point $$q_1^1 = (1,1,0,-1/3,1/3).$$ 

In order to obtain functions which give bounds for the coordinates $z_1$, $w_1$, $v_1$, we integrate system~(\ref{sys}) with controls described in Theorem~\ref{thmsing3}.     
The parabola  
\begin{align}
z_{\max} (x_1) = & \frac{1-x_1^2}{4},  \label{zm}
\end{align}
gives the following bound:
\begin{align}
|z_1|\leq z_{\max} (x_1). \label{restrz}
\end{align}

The projection of section of the attainable set $\mathcal{A}_{q_0}^{\mathrm{sing}} (1) \cap \{ y = 1\}$ to the right  half-plane $(x_1,z_1)$ is shown on left Fig.~\ref{at-xzw}. Notice that solutions for $z_1 = \pm z_{\max}(x_1)$ are unique and are given by trajectories with one switching for control $u_1$ in the set $\{-1,1\}$. The same solutions form the boundary for the projection of the attainable set on the plane $(x_1,w_1)$, see right Fig.~\ref{at-xzw}.

The following parabolas (with fixed values of $x_1$)
\begin{align}
w_{\max} (x_1, z_1) = & \frac{1}{96} \left(3 - 15 x_1 - 3 x_1^2 - 17 x_1^3\right) + \frac{z_1}{2} - \frac{z_1^2}{2 (1 + x_1)} \label{wm}
\end{align}
define the next boundary: 
\begin{align}
-w_{\max}(-x_1,-z_1) \leq w_1 \leq w_{\max}(x_1,z_1). \label{restrw}
\end{align}
Projections to the plane $(z_1,w_1)$ of sections of the attainable set with fixed values of $x_1$ are shown on Fig.~\ref{at-zw}. Solutions for $w_1 = w_{\max}(x_1,z_1)$ or $w_1 = -w_{\max}(-x_1,-z_1)$ are unique and given by trajectories with two switchings for control $u_1$ in the set $\{-1,1\}$.


Finally, we describe the last boundary for the attainable set $\mathcal{A}_{q_0}^{\mathrm{sing}} (1)$:
\begin{align}
v_{\min} (x_1, z_1, w_1)  \leq   v_1  \leq  v_{\max} (x_1, z_1, w_1) \label{restrv}
\end{align}
by the following functions:
\begin{align*}
w_{mm} (x_1, z_1) = & \frac{1}{6} \left(-x_1 (1 + x_1^2) + (3 + \sgn z_1) z_1 - \frac{4 z_1^2}{1 + x_1}\right), \\
v_{\max} (x_1, z_1, w_1) = & \frac{1}{12}\big(1 + 3 x_1^2 - 6 (1 - x_1) z_1\big)  \\
&\quad +
   \frac{\sqrt{2}}{48} \sqrt{
    9 \big(1 - x_1^2 + 4 z_1\big)^3 + 8 \big(12 w_1 + 3 x_1 + x_1^3 - 6 (1 - x_1) z_1\big)^2}, \\
v_{\min+} (x_1, z_1, w_1) = &\frac{1}{12} \left(3 + 12 w_1 + 3 x_1^2 + 2 x_1^3 - 6 (1 - x_1) z_1\right) \\
  &\quad -\frac{1}{12} (1 - x_1) \sqrt{(1 - x_1) (1 + 24 w_1 + 3 x_1 + 4 x_1^3) - 
   12 (1 - x_1) z_1 - 12 z_1^2}, \\
v_{\min-} (x_1, z_1, w_1) = &\frac{1}{12} \bigg(1 + 3 x_1^2 + 6 z_1 + 6 x_1 z_1  +   \Big(\big(1 - 12 w_1 - 2 x_1 - x_1^2 - 2 x_1^3 + 2 (1 + x_1) z_1\big)^2 \\
   &\quad+  4 \big(1 - x_1^2 - 4 z_1\big) \big((1 + x_1) (6 w_1 + x_1 + x_1^3) - 
       4 (1 + x_1) z_1 + 4 z_1^2\big)\Big)^{1/2}\bigg), \\
v_{\min} (x_1, z_1, w_1) = &\left\{ \begin{array}{ll}
         v_{\min+} (-x_1,-z_1,-w_1)  & \mbox{if $w_1 \geq w_{mm} (x_1, z_1)$};\\
		  	 v_{\min+} (x_1,z_1,w_1)  & \mbox{if $w_1 \leq -w_{mm} (-x_1, -z_1)$};\\
         v_{\min-} (x_1 \sgn z_1, |z_1|, w_1 \sgn z_1)& \mbox{if $-w_{mm} (-x_1, -z_1)\leq w_1 \leq w_{mm} (x_1, z_1)$}.\end{array} \right. 
\end{align*}
Notice that functions $w_{\max} (x_1, z_1)$ and $w_{mm} (x_1, z_1)$ are well defined near point $x_1 = -1$ since $\lim\limits_{x_1\to-1+0} \frac{z_1^2}{1+x_1} \to 0$. 


\begin{corollary}
The attainable set $\mathcal{A}_{q_0}^{\mathrm{sing}} (T) = \{ (x_1, y_1, z_1, v_1, w_1)\}$ has the following description:
\begin{align}\label{attain}
\left\{ \begin{array}{l}
|x_1 y_1| \leq T, \\
|z_1| \leq T^2 z_{\max} (x_1 y_1), 
\\
\left[ \begin{array}{l}
         \left\{ \begin{array}{l} 
			|y_1|=T, \\
  -T^3 w_{\max} (-x_1,-z_1) \leq w_1 \leq T^3 w_{\max} (x_1,z_1), \\
  T^3 v_{\min} (x_1,z_1,w_1) \leq y_1 v_1 \leq T^3 v_{\max} (x_1,z_1,w_1); 
				 \end{array} \right. \\
		  	 \left\{ \begin{array}{l} 
 |x_1|=T, \\
  -T^3 w_{\max} (y_1,-z_1) \leq v_1 \leq T^3 w_{\max} (-y_1,z_1), \\
  -T^3 v_{\max} (y_1,-z_1,-v_1) \leq x_1 w_1 \leq -T^3 v_{\min} (y_1,-z_1,-v_1).
\end{array} \right.				\end{array} \right.
				\end{array} \right.
\end{align}
\end{corollary}
\begin{proof}
Follows from the obtained bounds~(\ref{restrx}),~(\ref{restrz}),~(\ref{restrw}),~(\ref{restrv}) for the case $y=1, x\geq 0$, the dilation symmetry 
\be{dilations}
(x,y,z,v,w,t) \mapsto (T x, T y, T^2 z, T^3 v, T^3 w, T t)
\ee
 and the following discrete symmetries:
\begin{align}
(x,y,z,v,w) \mapsto (-x,y,-z,v,-w), \\
(x,y,z,v,w) \mapsto (x,-y,-z,-v,w), \\
(x,y,z,v,w) \mapsto (y,x,-z,-w,-v). 
\end{align}
\end{proof}

\begin{corollary}
The attainable set $\mathcal{A}_{q_0}^{\mathrm{sing}} (T)$ is semi-algebraic. 
\end{corollary}
\begin{proof}
Consider the section $y_1=1$ of the set $\mathcal{A}_{q_0}^{\mathrm{sing}} (1)$. As it was shown above it is described by inequalities~(\ref{restrx}), (\ref{restrz}), (\ref{restrw}), (\ref{restrv}). 

It is easy to see that inequalities~(\ref{restrx}),(\ref{restrz}),(\ref{restrw}) are equivalent to polynomial ones. Moreover, it is possible to get rid of the square root function in~(\ref{restrv}) by applying the following equivalences:
\begin{align*}
\sqrt{f(q)}\leq g(q) \Longleftrightarrow \left\{ \begin{array}{ll}
         f(q) \geq 0, &\\
		  	 g(q) \geq 0, & \\
         f(q) \leq g^2(q). &\end{array} \right. \qquad \qquad 
\sqrt{f(q)}\geq g(q) \Longleftrightarrow \left[ \begin{array}{ll}
         \left\{ \begin{array}{l} g(q) < 0, \\
						f(q) \geq 0;
				 \end{array} \right. \\
		  	 \left\{ \begin{array}{l} g(q) \geq 0,\\
         f(q) \geq g^2(q).
				 \end{array} \right.				\end{array} \right. 
\end{align*}
Therefore the section $y=1$ of the set $\mathcal{A}_{q_0}^{\mathrm{sing}} (1)$ is semi-algebraic, which proves via action of the symmetries that the whole attainable set  $\mathcal{A}_{q_0}^{\mathrm{sing}} (1)$ is semi-algebraic as well. Via dilations~\eq{dilations}, any attainable set $\mathcal{A}_{q_0}^{\mathrm{sing}} (T)$, $T > 0$, is semi-algebraic as well.
\end{proof}

\begin{remark}
Since singular trajectories are optimal, the attainable set $\mathcal{A}_{q_0}^{\mathrm{sing}} (T)$ is exactly the part of the radius $T$ sub-Finsler sphere filled by singular trajectories.
\end{remark}

\section{Bang-bang flow}
In this section we consider extremals that satisfy the condition $\card \{ t \in [0, T] \mid h_1 h_2 (\lambda_t) \ne 0 \} < \infty$. This analysis obviously applies to bang-bang arcs $\lambda_t, t \in (\alpha, \beta) \subset [0, T]$.

If $h_1 h_2 (\lambda_t) |_{(\alpha, \beta)} \ne 0$, then $u(t)|_{(\alpha, \beta)} = (s_1, s_2)$, $s_i = \sgn h_i(\lambda_t)|_{(\alpha, \beta)}$.
Thus bang-bang extremals satisfy, at the points where $h_1 h_2 (\lambda_t) \ne 0$, the following bang-bang Hamiltonian system with the Hamiltonian function $H = |h_1| + |h_2|$:
\begin{equation}\label{Ham_bang}
\begin{cases} 
\dot{h}_1 = -s_2 h_3, \\
\dot{h}_2 = s_1 h_3, \\
\dot{h}_3 = s_1 h_4 + s_2 h_5, \\
\dot{h}_4 = \dot{h}_5 = 0, \\
\dot{q} = s_1 X_1 + s_2 X_2.
\end{cases}
\end{equation}
Solutions to this system are piecewise polynomial:
\begin{itemize}
\item $h_3, x, y$ are piecewise linear,
\item $h_1, h_2, z$ are piecewise quadratic,
\item $v, w$ are piecewise cubic.
\end{itemize}
\begin{remark}
The mapping $(\lambda, q) \mapsto (k\lambda, q), k> 0$, preserves extremal trajectories, thus in the sequel we consider only the reduced case 
\begin{gather*}
H(\lambda) \equiv 1.
\end{gather*}
\end{remark}

Let us parameterize the square
\be{H=1} \{ (h_1, h_2) \in \mathbb{R}^2  \mid  |h_1| + |h_2| = 1 \}  
\ee
by an angular coordinate $\theta \in S^1 = \mathbb{R}/2\pi \mathbb{Z}$:
$$h_1 = \sgn (\cos \theta) \cos^2 \theta, \quad h_2 = \sgn (\sin \theta) \sin^2 \theta.$$

With the use of the coordinate $\theta$, the vertical part of Hamiltonian system \eq{Ham_bang} reduces to the following system:
\begin{equation}\label{Hamtheta}
\begin{cases} 
\dot{\theta} = \frac{h_3}{|\sin 2\theta|}, \quad \theta \ne \frac{\pi n}{2}, \\
\dot{h}_3 = s_1 h_4 +s_2 h_5, \quad s_1 = \sgn \cos \theta, \quad s_2  = \sgn \sin \theta.
\end{cases}
\end{equation}
We have proved the following statement.

\begin{prop}\label{prop:Hamtheta}
Bang-bang extremal arcs satisfy ODE \eq{Hamtheta}.
\end{prop}

Along bang-bang arcs the function $\theta (t)$ is continuous, and smooth for $\theta \ne \frac{\pi n}{2}$.

\subsection{Discrete symmetries of the Hamiltonian system} \label{subsec:symmetr}
Consider the following mappings of the dual of the Lie algebra $L^* \cong \mathbb{R}_{h_1 \dots h_5}^5$:
\begin{align*}
&\varepsilon^1 : (h_1, h_2, h_3, h_4, h_5) \mapsto (h_2, h_1, -h_3, -h_5, -h_4),\\
&\varepsilon^2 : (h_1, h_2, h_3, h_4, h_5) \mapsto (-h_2, -h_1, -h_3, h_5, h_4),\\
&\varepsilon^3 : (h_1, h_2, h_3, h_4, h_5) \mapsto (h_1, -h_2, -h_3, -h_4, h_5),\\
&\varepsilon^4 : (h_1, h_2, h_3, h_4, h_5) \mapsto (-h_1, h_2, -h_3, h_4, -h_5),\\
&\varepsilon^5 : (h_1, h_2, h_3, h_4, h_5) \mapsto (-h_2, h_1, h_3, -h_5, h_4),\\
&\varepsilon^6 : (h_1, h_2, h_3, h_4, h_5) \mapsto (h_2, -h_1, h_3, h_5, -h_4),\\
&\varepsilon^7 : (h_1, h_2, h_3, h_4, h_5) \mapsto (-h_1, -h_2, h_3, -h_4, -h_5).\\
\end{align*}
The following statement is verified immediately.

\begin{prop}\label{propos:symmetr}
The mappings $\varepsilon^i$, $i = 1, \dots, 7$, generate the group $G = \{ \Id, \varepsilon^1, \dots, \varepsilon^7 \}$ with the product Table \ref{tab:symmetr}. These mappings are symmetries of system \eq{Hamtheta}, i.e., transform its solutions into solutions.
\end{prop}

\begin{table}[h]
\caption{Product rule in group $G$: line $i$, column $j$ contains $\varepsilon^i \circ \varepsilon^j$}\label{tab:symmetr}
\begin{center}
\begin{tabular}{c|c|c|c|c|c|c|c|}
 & $\varepsilon^1$ & $\varepsilon^2$ & $\varepsilon^3$ & $\varepsilon^4$ & $\varepsilon^5$ & $\varepsilon^6$ & $\varepsilon^7$ \\
\hline
$\varepsilon^1$ & $\Id$ & $\varepsilon^7$ & $\varepsilon^6$& $\varepsilon^5$ & $\varepsilon^4$ & $\varepsilon^3$ & $\varepsilon^2$ \\
\hline
$\varepsilon^2$ & $\varepsilon^7$ & $\Id$ & $\varepsilon^5$ & $\varepsilon^6$ & $\varepsilon^3$ & $\varepsilon^4$ & $\varepsilon^1$ \\
\hline
$\varepsilon^3$ & $\varepsilon^5$ & $\varepsilon^6$ & $\Id$ & $\varepsilon^7$ & $\varepsilon^1$ & $\varepsilon^2$ & $\varepsilon^4$\\
\hline
$\varepsilon^4$ & $\varepsilon^6$ & $\varepsilon^5$ & $\varepsilon^7$ & $\Id$ & $\varepsilon^2$ & $\varepsilon^1$ & $\varepsilon^3 $\\
\hline
$\varepsilon^5$ & $\varepsilon^3$ & $\varepsilon^4$ & $\varepsilon^2$ & $\varepsilon^1$ & $\varepsilon^7$ & $\Id$ & $\varepsilon^6$\\
\hline
$\varepsilon^6$ & $\varepsilon^4$ & $\varepsilon^3$ & $\varepsilon^1$ & $\varepsilon^2$ & $\Id$ & $\varepsilon^7$ & $\varepsilon^5$\\
\hline
$\varepsilon^7$ & $\varepsilon^2$ & $\varepsilon^1$ & $\varepsilon^4$ & $\varepsilon^3$ & $\varepsilon^6$ & $\varepsilon^5$ & $\Id$\\
\hline
\end{tabular}
\end{center}
\end{table}

The group $G$ is isomorphic to the group of symmetries of the square $\{ H = |h_1| + |h_2| = 1 \} : \varepsilon^1, \varepsilon^2$ are reflections in middle perpendiculars of sides; $\varepsilon^3, \varepsilon^4$ are reflections in diagonals;
$\varepsilon^5, \varepsilon^6, \varepsilon^7$ are respectively rotations by $\pi/2, -\pi/2, \pi$.


Any point of the plane $(h_4, h_5)$ can be transformed to a point of the angle $\Omega = \{ (h_4, h_5) \in \mathbb{R}^2 \mid h_4 \ge h_5 \ge 0 \}$. The angle $\Omega$ is a fundamental domain of the action of the group $G$ in the plane $(h_4, h_5)$.
Thus in the study of system \eq{Hamtheta} we can restrict ourselves by the case $(h_4, h_5) \in \Omega$. This case obviously decomposes into the following subcases:
\begin{enumerate}
\item[$1)$] $h_4 > h_5 >0$,
\item[$2)$] $h_4 > h_5 = 0$,
\item[$3)$] $h_4 = h_5 > 0$,
\item[$4)$] $h_4 = h_5 = 0$.
\end{enumerate}
\begin{remark}
One checks immediately that the Casimir $E$ is preserved by the symmetry group $G: E \circ \varepsilon^i = E, \quad i = 1, \dots, 7$.
\end{remark}
\subsection{Phase portrait of system \eq{Hamtheta}}
We consider system \eq{Hamtheta} as an oscillator, with the full energy
$$ E = \frac{h_3^2}{2} + h_1 h_5 - h_2 h_4 = \frac{h_3^2}{2} + U(\theta)$$
and the potential energy
$$U(\theta) = h_1 h_5 - h_2 h_4 = s_1 \cos^2 \theta h_5 - s_2 \sin^2 \theta h_4.$$
The function $U(\theta)$ is $C^1$-smooth at $\theta = \frac{\pi n}{2}$ and analytic elsewhere.

\subsubsection{Case $1)$: $h_4 > h_5 >0$}
The plot of the potential energy $U(\theta)$ is given in Fig.~\ref{fig:1)U}.

\twofiglabel{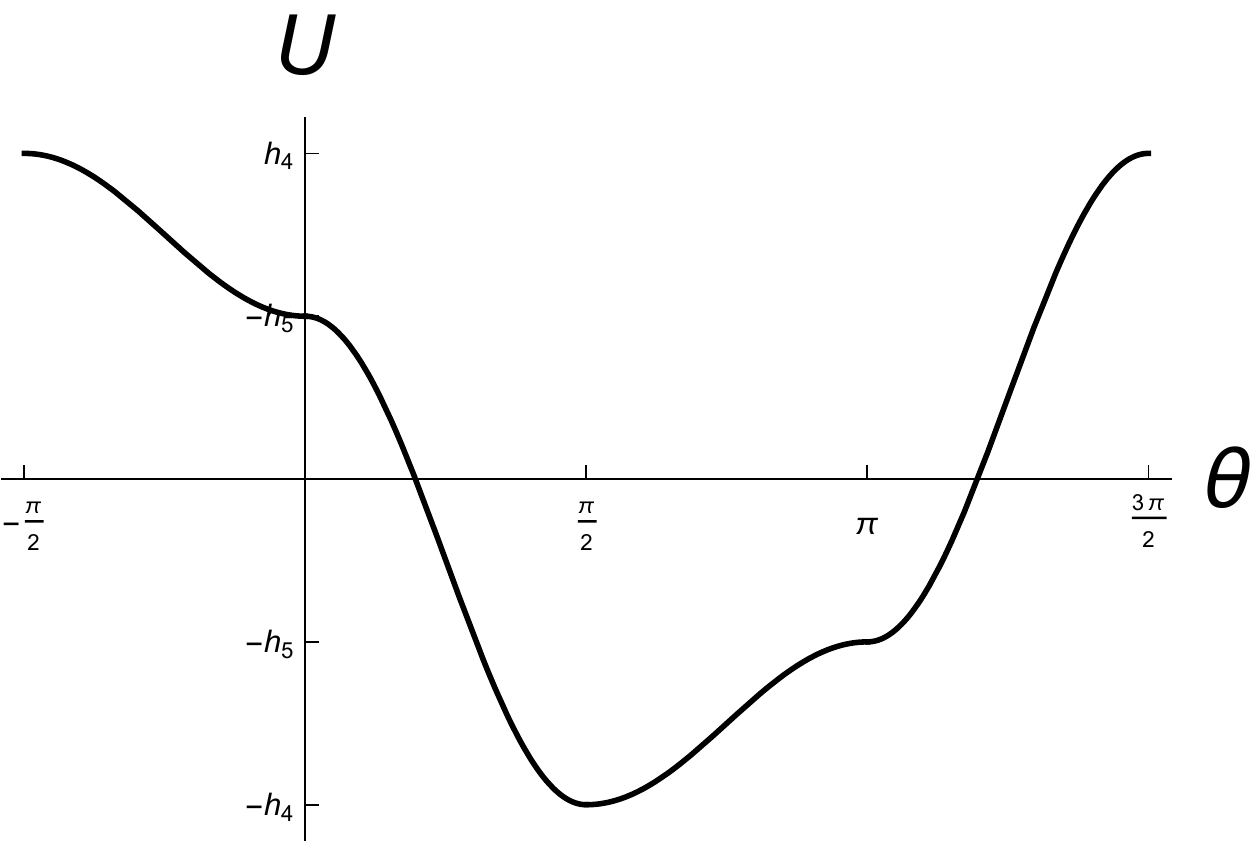}{Plot of $U(\theta)$ in case 1)}{fig:1)U}
{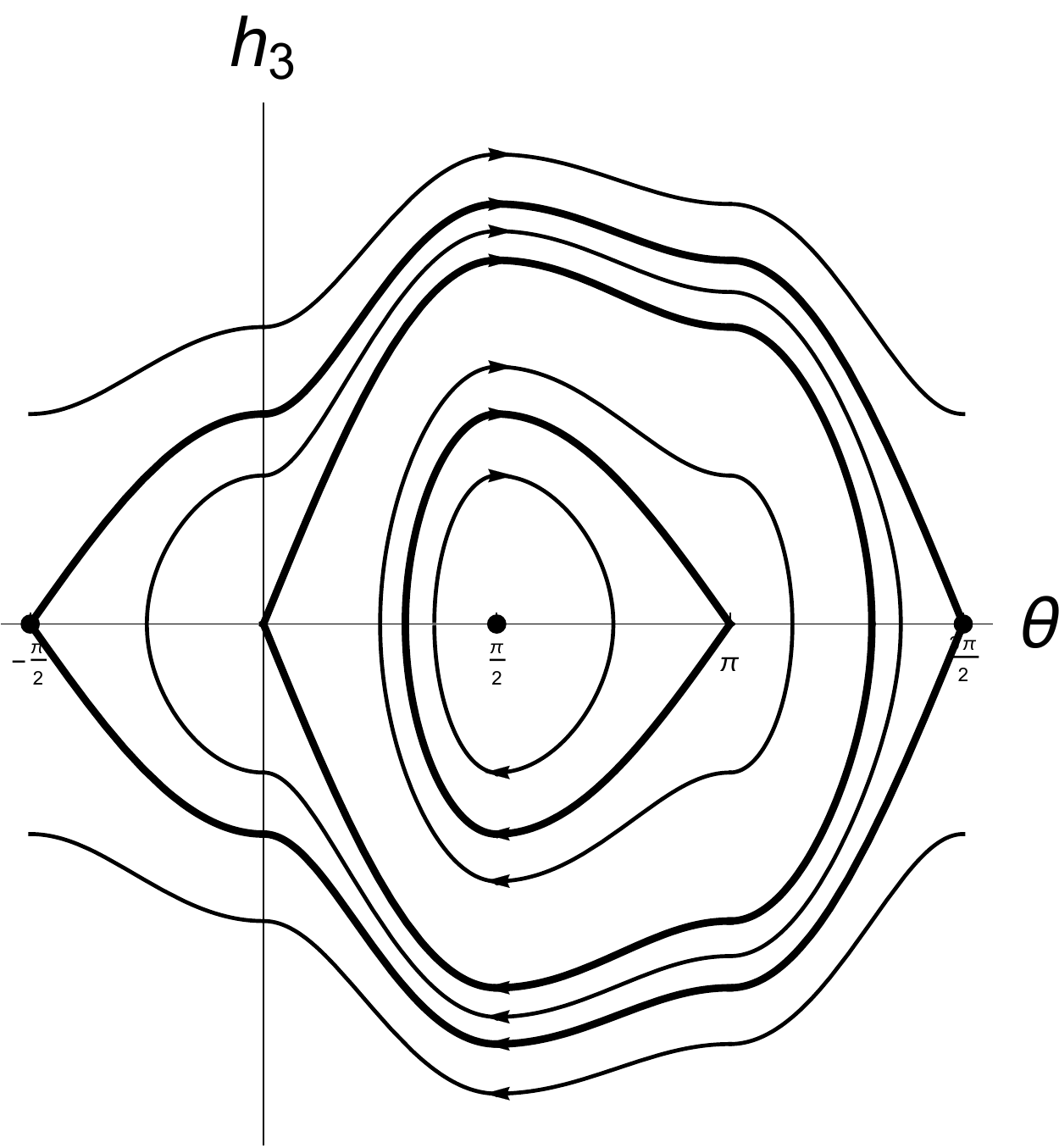}{Phase portrait of system \eq{Hamtheta} in case 1)}{fig:1)h3th}

Then the phase portrait of system \eq{Hamtheta} is drawn as a set of curves $h_3 = \pm \sqrt{2(E - U(\theta))}$, see Fig.~\ref{fig:1)h3th}.

The critical level lines of the energy $E$ are:
\begin{align*}
&C_1 = E^{-1} (-h_4),  &&C_3 = E^{-1} (-h_5),\\
&C_5 = E^{-1} (h_5),  &&C_7 = E^{-1} (h_4).
\end{align*}
The domains of regular values of energy $E$ are:
\begin{align*}
&C_2 = E^{-1} (-h_4, -h_5),  &&C_6 = E^{-1} (h_5, h_4), \\
&C_4 = E^{-1} (-h_5, h_5),  &&C_8 = E^{-1} (h_4, +\infty).
\end{align*}
Thus we get a decomposition of a section of the cylinder $ C = L^* \cap  \{ H = 1 \}$:
$$ \{ \lambda \in C \mid h_4 > h_5 > 0 \} = \cup_{i=1}^8 C_i.$$
All level lines of the energy $E$ are analytic for $\theta \ne \frac{\pi n}{2}$ and $C^1$-smooth for $\theta = \frac{\pi n}{2}, h_3 \ne 0$. The critical level lines have corners at $\theta = \frac{\pi n}{2} \ne \frac{\pi}{2}, h_3=0$. The level line $C_1$ is just a point $(\theta, h_3) = (\frac{\pi}{2}, 0)$.

\begin{remark}
Let $h_4 > h_5 > 0$ and $\lambda \in C \backslash C_7$. Then for any $T > 0$ there exists a unique bang-bang extremal $\lambda_t$, $t \in [0, T]$, with $\lambda_0 = \lambda$, and correspondingly a unique extremal trajectory $q(t) = \pi (\lambda_t)$, which we denote as $\Exp(\lambda, t)$, $t \in [0, T]$.

Further, let $h_4 > h_5 > 0$ and $\lambda \in C_7$. Then for any $T > 0$ there exists a finite number of bang-bang extremals $\{ \lambda_t^1, \dots, \lambda_t^N \}$, $t \in [0, T]$, with $\lambda_0^i = \lambda$, and correspondingly a finite number of bang-bang extremal trajectories $\{ q^1(t), \dots, q^N(t) \} = \{ \pi (\lambda_t^1), \dots, \pi(\lambda_t^N) \} =: \Exp(\lambda, t)$. Namely, for $h_4 > h_5 > 0$ and $\lambda \in C_7$ 
an extremal arc $\lambda_t, t\in (T-\varepsilon, \tau]$, splits at the points $(\theta, h_3) = (\frac{3\pi}{2}, 0)$ into two extremal arcs: $\lambda_t^+$ --- for which $h_3(\lambda_t^+) > 0$ for $t \in (\tau, \tau + \varepsilon)$, and $\lambda_t^-$ --- for which $h_3 (\lambda_t^-) < 0$ for $t \in (\tau, \tau + \varepsilon)$.

Summing up, in case $1)$: $h_4 > h_5 > 0$ we can define an exponential mapping for bang-bang trajectories: $C \times \mathbb{R}_+ \ni (\lambda, t) \mapsto \Exp (\lambda, t) \subset M$, single-valued for $\lambda \in C \backslash C_7$ and set-valued for $\lambda \in C_7$.
\end{remark}

\subsubsection{Case $2)$: $h_4 > h_5 = 0$}
The plot of the potential energy $U(\theta)$ and the corresponding phase portrait of system \eq{Hamtheta} are given respectively in Fig.~\ref{fig:2)U} and in Fig.~\ref{fig:2)h3th}.
\twofiglabel
{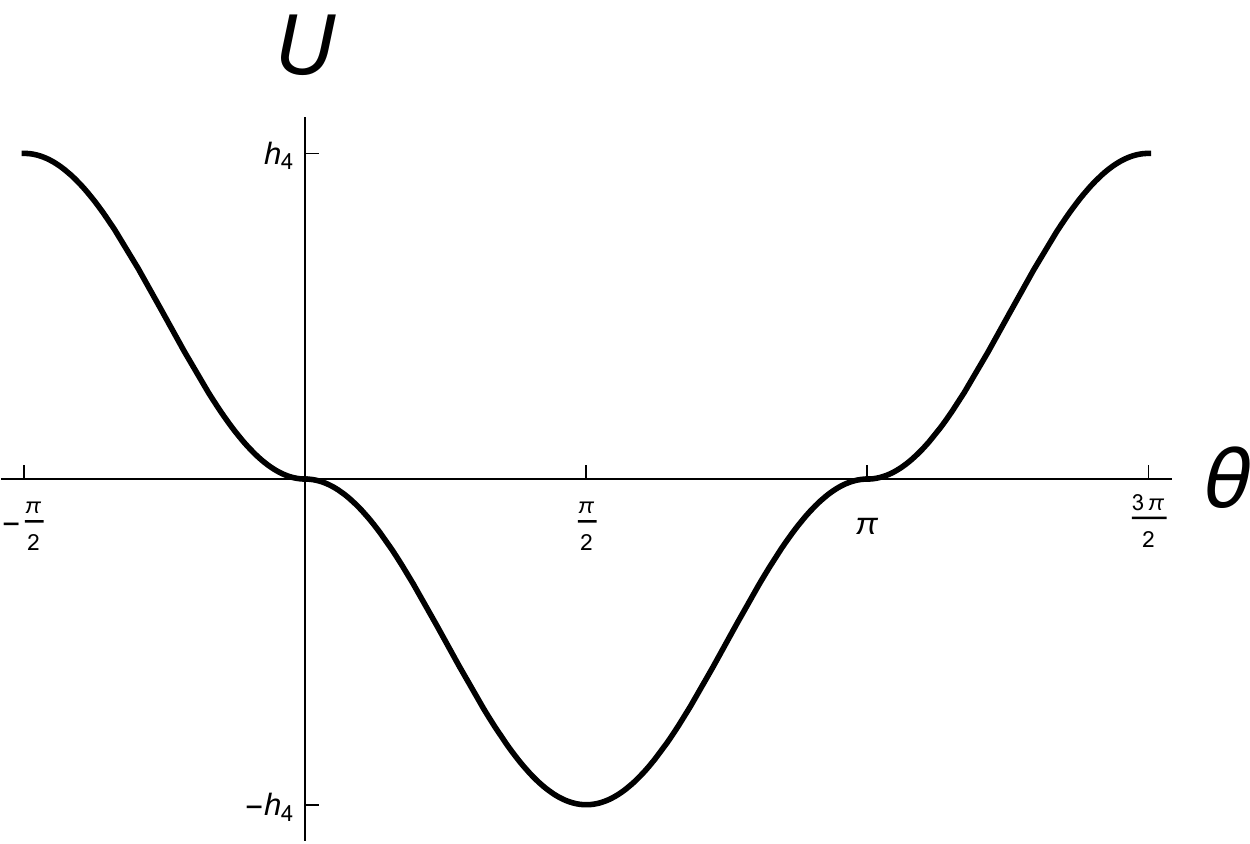}{Plot of $U(\theta)$ in case 2)}{fig:2)U}
{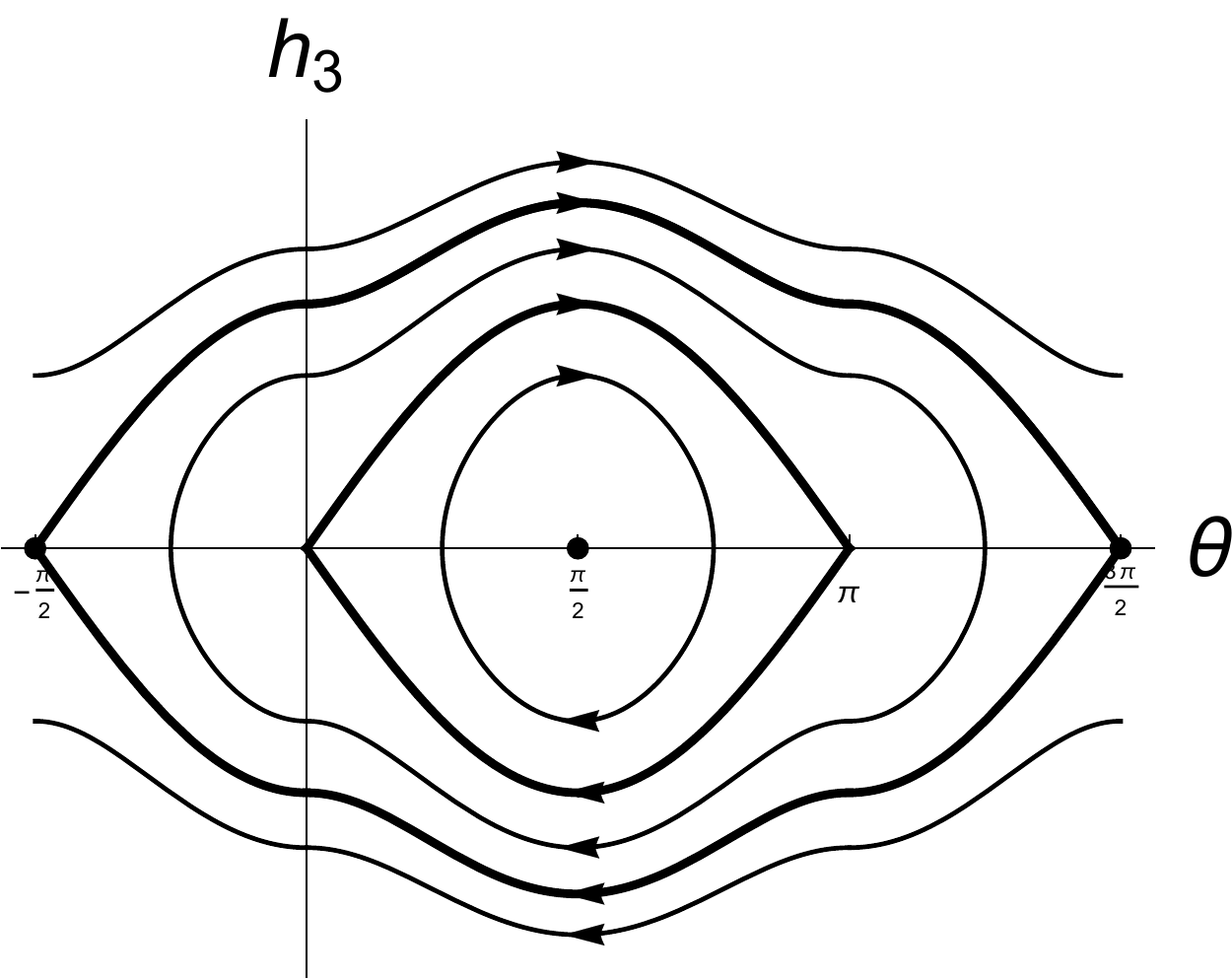}{Phase portrait of system \eq{Hamtheta} in case 2)}{fig:2)h3th}

The phase portrait contains the critical level lines of the energy $E$:
$$ C_1 = E^{-1} (-h_4), \quad C_3 = E^{-1} (0),  \quad  C_5 = E^{-1} (h_4), $$
and the domains of regular values of $E$:
$$ C_2 = E^{-1} (-h_4, 0), \quad C_4 = E^{-1} (0, h_4), \quad C_6 = E^{-1} (h_4, +\infty). $$
Thus we have a decomposition:
$$
\{ \lambda \in C \mid h_4 > h_5 = 0 \} = \cup_{i=1}^6 C_i.
$$
\begin{remark}
Similarly to case $1)$, we can define a single-valued exponential mapping for $h_4 > h_5 = 0$, $E \ne h_4$, and a set-valued exponential mapping for $h_4 > h_5 = 0$, $E = h_4$.
\end{remark}

\subsubsection{Case $3)$: $h_4 = h_5 > 0$}
The plot of $U(\theta)$ and the phase portrait of \eq{Hamtheta} are given in Fig.~\ref{fig:3)U} and in Fig.~\ref{fig:3)h3th} respectively.
\twofiglabel
{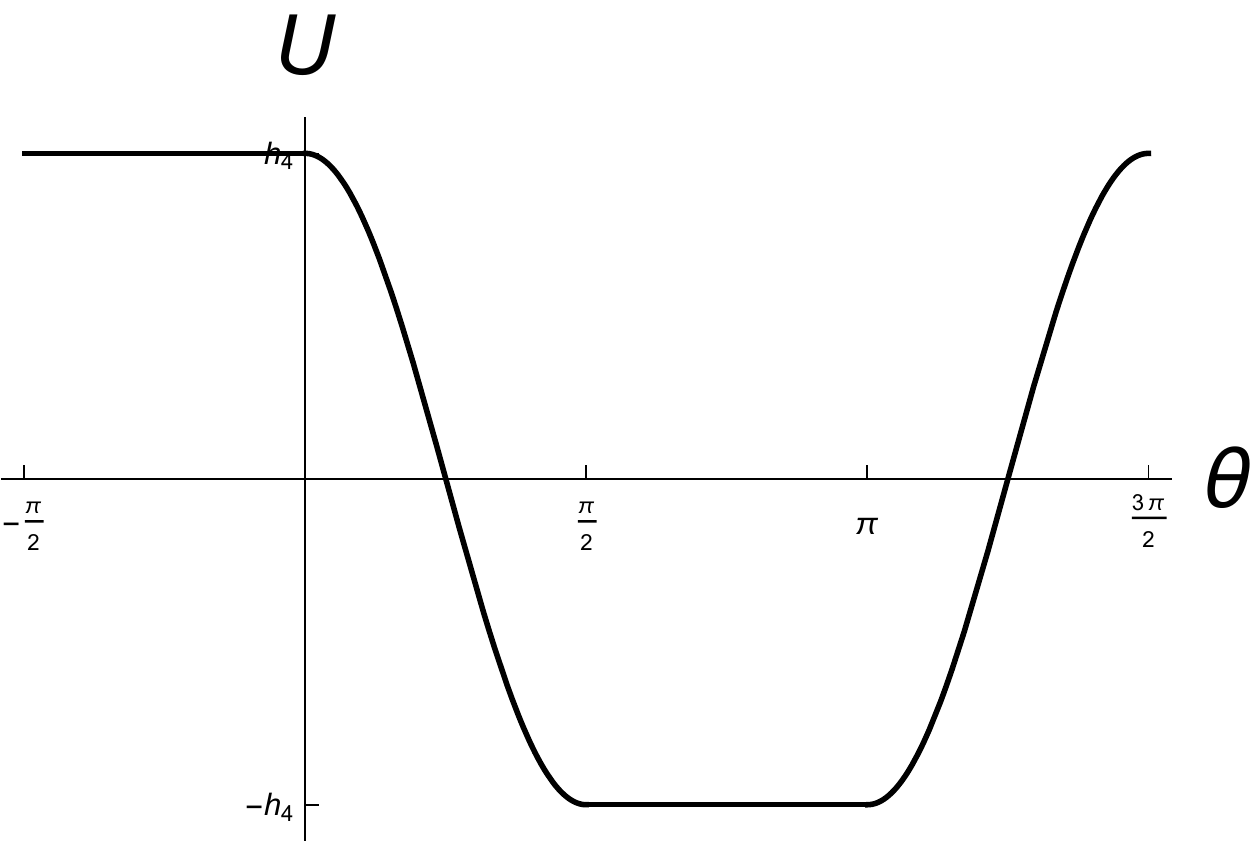}{Plot of $U(\theta)$ in case 3)}{fig:3)U}
{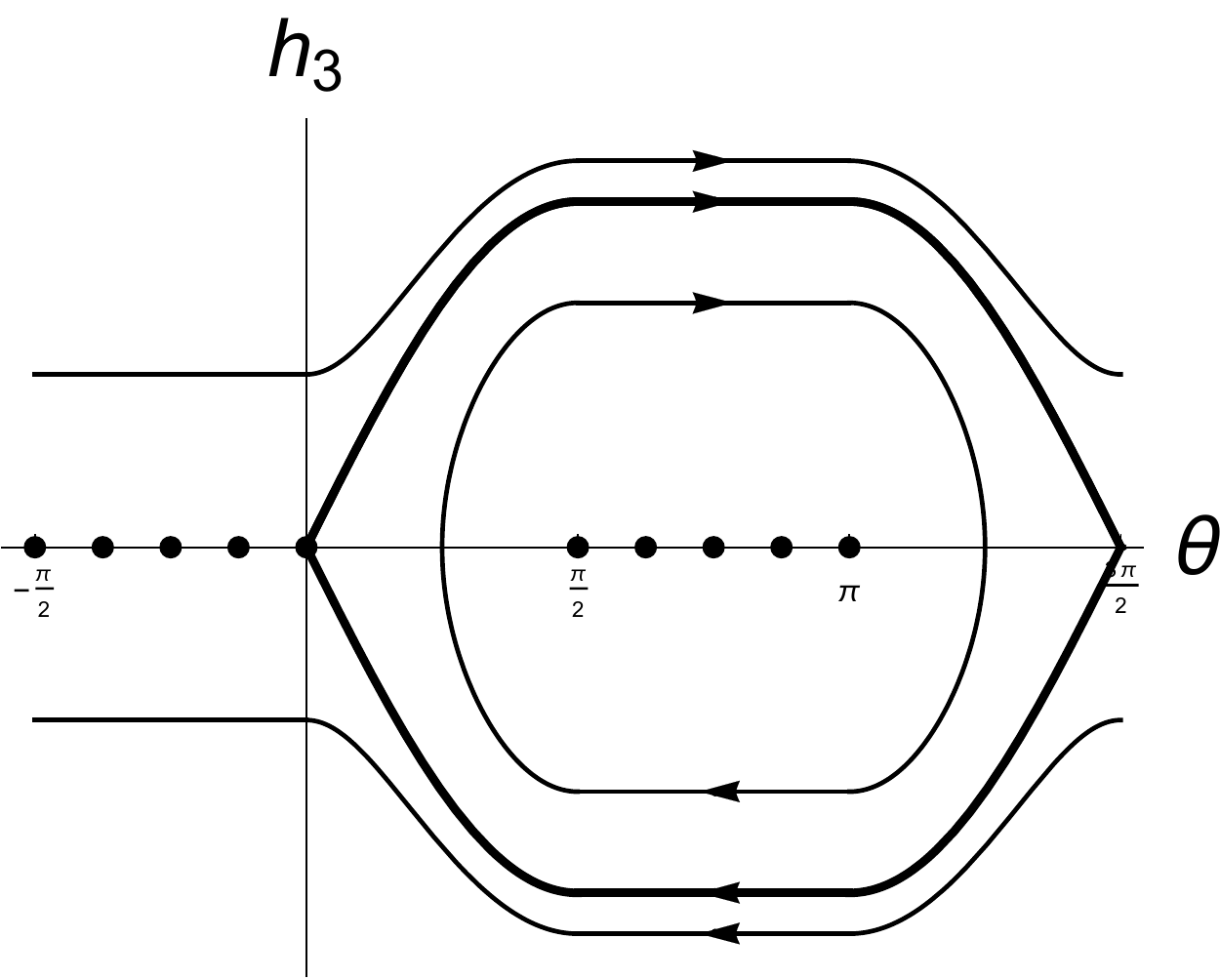}{Phase portrait of system \eq{Hamtheta} in case 3)}{fig:3)h3th}


The phase portrait contains the critical level lines of the energy $E$:
$$ C_1 = E^{-1} (-h_4),   \quad C_3 = E^{-1} (h_4), $$
and the domains of regular values of $E$:
$$ C_2 = E^{-1} (-h_4, 0), \quad C_4 = E^{-1}   (h_4, +\infty). $$
Thus we have a decomposition:
$$
\{ \lambda \in C \mid h_4 = h_5 > 0 \} = \cup_{i=1}^4 C_i.
$$

\begin{remark}
Similarly to case $1)$, a single-valued exponential mapping is defined for $h_4 = h_5 > 0$.
\end{remark}

\subsubsection{Case $4)$: $h_4 = h_5 = 0$}
Finally in case $4)$ we have $U(\theta) \equiv 0$, see Fig.~\ref{fig:4)U}, and the phase portrait of \eq{Hamtheta}  is in  Fig.~\ref{fig:4)h3th}.
\twofiglabel
{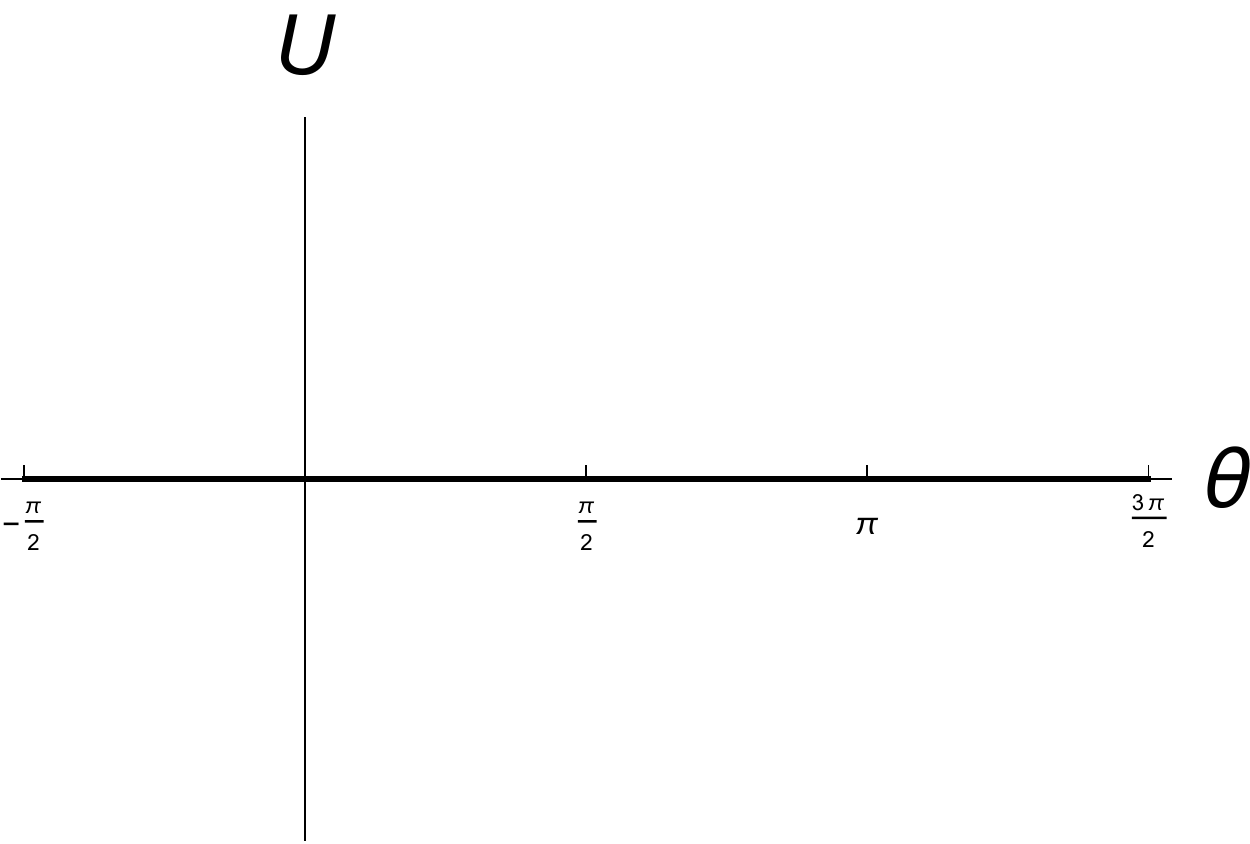}{Plot of $U(\theta)$ in case 4)}{fig:4)U}
{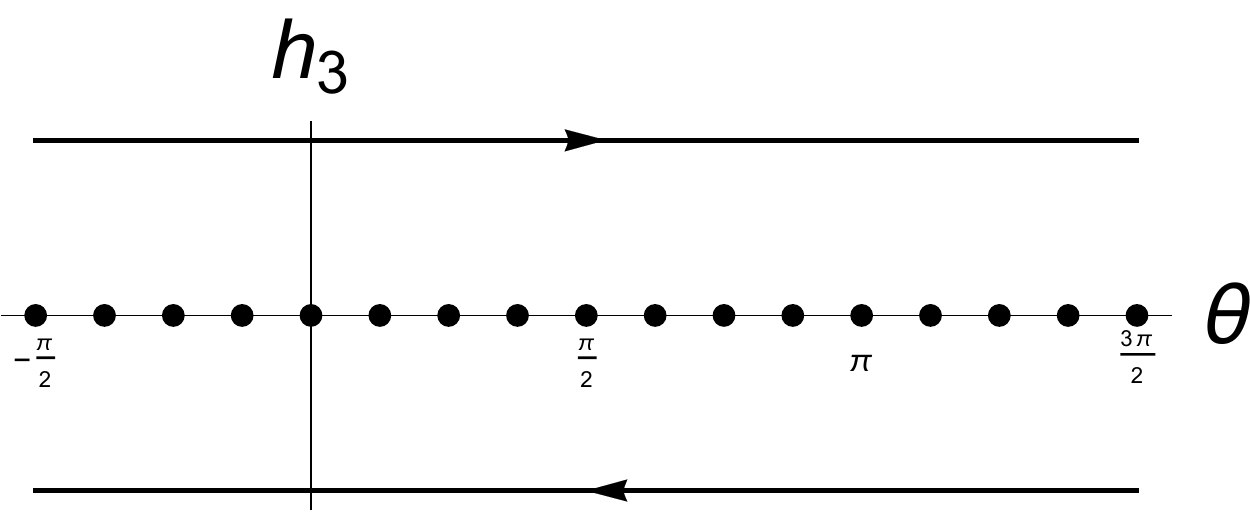}{Phase portrait of system \eq{Hamtheta} in case 4)}{fig:4)h3th}

The critical level line $C_1 = E^{-1}(0)$ consists of fixed points, and the domain of regular values of energy is $C_2 = E^{-1}(0, +\infty)$. We have
$$
\{ \lambda \in C \mid h_4 = h_5 = 0 \} = C_1 \cup C_2.
$$

\begin{remark}
Let $\lambda \in C$, $h_4 = h_5 = 0$. Then for any $T > 0$ there is uniquely defined a solution $\lambda_t$ of \eq{Hamtheta} with $\lambda_0 = \lambda$, and respectively there is a uniquely defined $q(t) = \pi(\lambda_t) = \Exp(\lambda, t)$.
\end{remark}

\subsection{Cut time along bang-bang trajectories}
Summing up remarks at the end of the previous four subsubsections, we get the following statement.

\begin{prop}
\label{propos:flow}
Let $\lambda \in C = L^* \cap \{ H=1 \}$, and let $h_4 \geq h_5 \geq 0$. If $E \ne h_4 > h_5$ or $h_4 = h_5$, then for any $t > 0$ there exists a unique solution $\lambda_t$ to system \eq{Hamtheta} with $\lambda_0 = \lambda$, and
respectively a unique bang-bang trajectory $q(t) = \pi(\lambda_t) =: \Exp(\lambda, t)$. If $E = h_4 > h_5$, then for any $T>0$ there exist a finite number of solutions $\{ \lambda_t^1, \dots, \lambda_t^N\}, t \in [0, T]$, to system \eq{Hamtheta}, with $\lambda_0^1 = \dots = \lambda_0^N = \lambda$, and respectively a finite number of bang-bang trajectories $\{ q^1(t), \dots, q^N(t) \} = \{ \pi(\lambda_t^1), \dots,  \pi(\lambda_t^N)\} =: \Exp(\lambda, t).$
\end{prop}

So there is a defined an exponential mapping for bang-bang trajectories, set-valued in the general case. Then we can define, similarly to sub-Riemannian geometry, the cut time along bang-bang trajectories: 
$$t_{\cut} := \sup \{ T> 0 \mid \text{ at least one of the trajectories } \Exp(\lambda, t) \text{ is optimal for } t \in [0, T] \}.
$$
By Lemma \ref{lem:ui=+-1}, small bang-bang arcs are optimal, i.e., $t_{\cut} (\lambda) > 0$ for any $\lambda \in C$.

\section{Conclusion}
In this paper we started a study of the $\ell_{\infty}$ sub-Finsler problem on the Cartan group. Many questions remain unsolved, e.g.:
\begin{itemize}
\item
optimality of bang-bang and mixed trajectories,
\item
uniform bounds on the number of smooth arcs of minimizers that connect points in the Cartan group,
\item
regularity of sub-Finsler distance and sphere.
\end{itemize}
We postpone study of these questions to forthcoming papers.

\section*{Acknowledgments}

The authors thank Prof. Andrei Agrachev and Prof. Lev Lokutsievsky for fruitful discussions of sub-Finsler geometry.

	\providecommand{\noopsort}[1]{#1}
\providecommand{\bysame}{\leavevmode\hbox to3em{\hrulefill}\thinspace}
\providecommand{\MR}{\relax\ifhmode\unskip\space\fi MR }
\providecommand{\MRhref}[2]{%
  \href{http://www.ams.org/mathscinet-getitem?mr=#1}{#2}
}
\providecommand{\href}[2]{#2}

\end{document}